\documentclass{article}
\usepackage[utf8]{inputenc}
\usepackage{amsmath, authblk}
\usepackage{amssymb,comment}
\usepackage{amsthm}
\usepackage{aligned-overset}
\usepackage{xcolor}
\usepackage{graphicx}
\usepackage{appendix}
\usepackage{hyperref}
\usepackage[shortlabels]{enumitem}

\allowdisplaybreaks
\theoremstyle{plain}
\newtheorem{theorem}{Theorem}[section]
\newtheorem{lemma}[theorem]{Lemma}
\newtheorem{proposition}[theorem]{Proposition}
\newtheorem{corollary}[theorem]{Corollary}
\theoremstyle{definition}
\newtheorem{definition}[theorem]{Definition}
\newtheorem{example}[theorem]{Example}
\newtheorem{remark}[theorem]{Remark}

\newcommand{\R}{\mathbb{R}}
\newcommand{\N}{\mathbb{N}}

\newcommand*\mathinhead[2]{\texorpdfstring{$\boldsymbol{#1}$}{#2}}

\usepackage[
  backend=biber,
  style=numeric,
  sorting=nyt,
  giveninits=true,
  maxnames=99,
  maxbibnames=99,
  doi=false,
  isbn=false,
  url=false,
  eprint=true
]{biblatex}

\DeclareFieldFormat[article,inproceedings,misc]{title}{#1}

\DeclareFieldFormat[article,inproceedings]{pages}{#1}

\DeclareBibliographyDriver{article}{
  \printnames{author}
  \newunit\newblock
  \printfield{title}
  \newunit\newblock
  \printfield{journaltitle}
  \setunit{\addcomma\space}
  \printfield{volume}
  \iffieldundef{number}
    {}
    {\mkbibparens{\printfield{number}}}
  \setunit{\addcolon\allowbreak}
  \printfield{pages}
  \setunit{\addcomma\space}
  \printfield{year}
  \finentry
}

\DeclareBibliographyDriver{inproceedings}{
  \printnames{author}
  \newunit\newblock
  \printfield{title}
  \newunit\newblock
  \printfield{booktitle}
  \setunit{\addcomma\space}
  \printfield{pages}
  \setunit{\addcomma\space}
  \printfield{year}
  \finentry
}

\DeclareBibliographyDriver{book}{
  \printnames{author}
  \newunit\newblock
  \printfield{title}
  \newunit\newblock
  \printfield{edition}
  \newunit\newblock
  \printlist{publisher}
  \setunit{\addcomma\space}
  \printfield{year}
  \finentry
}

\title{$\mathcal{P}$-Sensitive Functions and Localizations\footnote{funded by the Deutsche Forschungsgemeinschaft (DFG, German Research Foundation) – 471178162}}
\author{Johannes Langner\footnote{johannes.langner@insurance.uni-hannover.de}  \, \& Gregor Svindland\footnote{gregor.svindland@insurance.uni-hannover.de}}
\affil{\small Institute for Actuarial and Financial Mathematics and House of Insurance\\ Gottfried Wilhelm Leibniz Universit\"at Hannover}
\date{25.01.2026}

\begin{document}

\maketitle

\begin{abstract}
    This paper assumes a robust stochastic model where a set $\mathcal{P}$ of probability measures replaces the single probability measure of dominated models. We introduce and study $\mathcal{P}$-sensitive functions defined on robust function spaces of random variables. We show that $\mathcal{P}$-sensitive functions are precisely those that admit a representation via so-called functional localization.
    The theory is applied to solving robust optimization problems, to convex risk measures, and to the study of no arbitrage in robust one-period financial models. 
\end{abstract}

\section{Introduction}

We consider a robust probabilistic model given by a measure space $(\Omega, \mathcal{F})$ and a non-empty set of probability measures $\mathcal{P}$ on $(\Omega, \mathcal{F})$. The set $\mathcal{P}$ describes the degree of ambiguity, or Knightian uncertainty, inherent in the model. This includes the case where $\mathcal{P} = \{P\}$, where there is no ambiguity, as well as the case in which there exists a probability measure $P$ dominating each element in $\mathcal{P}$ and ambiguity may be present. In both cases, the mathematical machinery used to handle such models typically relies on the dominating probability measure $P$, which, among other things, determines relevant model spaces of random variables and provides exhaustion and approximation methods. If $\mathcal{P}$ is not dominated by a probability measure, then the aforementioned approaches and techniques fail, and we are in a truly robust setting. It is this latter situation that we have in mind throughout this study, even though we do not exclude the other cases. Some prominent examples of robust stochastic models and robust model spaces considered in financial mathematics are the volatility uncertainty models studied in, e.g., \cite{BKN2021}, \cite{C2012}, \cite{MPZ2013}, \cite{NS2012}, \cite{STZ2011}, \cite{STZ2012}, and \cite{STZ2013}, as well as the models applied to study the Fundamental Theorem of Asset Pricing and the superhedging problem in \cite{BKN2020}, \cite{BC2020}, \cite{BN2015}, \cite{BFH2019}, \cite{BM2020}, \cite{CFR2022}, and \cite{HO2018}. For further references, we refer to the references in those articles.

Let $(\Omega, \mathcal{F}, \mathcal{P})$ be a robust probabilistic model and define the upper probability given by $\mathcal{P}$ as
\begin{equation*}
    c(A):=\sup_{P\in \mathcal{P}}P(A), \quad A\in \mathcal{F}.
\end{equation*}
The robust function space $L^0_c$ consists of equivalence classes of random variables, which are identified up to $P$-almost sure equality under each $P\in \mathcal{P}$ (see Section~\ref{sec:prelim:not}). If $\mathcal{P}=\{P\}$, then, of course, $L^0_c = L^0_P := L^0(\Omega, \mathcal{F}, P)$. 

In this paper, we consider functions $f$ mapping subsets $\mathcal{X}$ of $L^0_c$ to the extended real numbers. Examples of such functions in financial mathematics include robust risk measures, the superhedging functional in robust financial market models, and robust utilities. The primary question we address is under which conditions such a function $f \colon \mathcal{X} \to [-\infty, \infty]$ can be represented as
\begin{equation} \label{eq:robust:rep}
    f(X)=\sup_{Q\in \mathcal{Q}}f^{Q}(X), \quad X\in \mathcal{X}.
\end{equation}
Here $\mathcal{Q}$ is a set of probability measures on $(\Omega,\mathcal{F})$, and $f^{Q} \colon \mathcal{X} \to [-\infty, \infty]$, $Q \in \mathcal{Q}$, are functions that are consistent with the respective $Q$ in the sense that, whenever $X, Y \in \mathcal{X}$ look identical under $Q$, that is, $Q(x = y) = 1$ for all representatives $x \in X$ and $y \in Y$, we have $f^{Q}(X) = f^{Q}(Y)$. Due to this consistency, each $f^{Q}$ may also be identified with a function on a subset of the reduced model space $L^0_Q$.

The existence of a \textit{robust representation} \eqref{eq:robust:rep} is very useful when handling robust models, since it allows to break down the mathematical reasoning for $f$ to the classical dominated level, that is, to argue for $f^{Q}$ under each $Q\in \mathcal{Q}$, and then to aggregate over the probability measures $Q \in \mathcal{Q}$. We show that the existence of a representation \eqref{eq:robust:rep} is equivalent to a property called \textit{$\mathcal{P}$-sensitivity} of $f$. For sets of (equivalence classes of) random variables, this property has already been studied, for instance, in \cite{BM2020}, \cite{LS2025}, and \cite{MMS2018}. We extend the definition of $\mathcal{P}$-sensitivity from sets to functions and illustrate its usefulness.

A family of functions $(f^{Q})_{Q \in \mathcal{Q}}$ such that \eqref{eq:robust:rep} holds, is called a \textit{functional localization} of $f$. This localization of $f$ is not unique, and there are several ways to construct a localizing family depending on the properties of the given function $f$. In particular, we study two canonical types of localizations. One is based on a primal reduction of the function $f$ given a probability measure $Q$ via its level sets, while the other takes a dual approach in the sense of the Fenchel-Moreau theorem. We derive conditions under which both localizations coincide.

For illustration, the theory is applied to optimization problems, convex risk measures, and the superhedging functional in a robust one-period financial market model. For $\mathcal{P}$-sensitive functions $f$ with localization $(f^{Q})_{Q \in \mathcal{Q}}$, we show how solving optimization problems for $f$ may be reduced to solving corresponding optimization problems for each $f^{Q}$ in a dominated framework and then aggregating the optimizers. As regards convex risk measures, we are interested in the interplay of a given risk measure defined on the robust model space under $\mathcal{P}$---which is understood as an aggregate of risk opinions---with the elements of its two canonical localizations, which may be interpreted as individual risk opinions under some probability measure $Q$. We show that there may be inconsistencies, which we call localization bubbles, when breaking down the risk measure to the local dominated level, and we provide conditions under which such localization bubbles do not appear. In our study of robust one-period financial market models, suitable localizations of the superhedging functional provide different robust versions of the Fundamental Theorem of Asset Pricing.

\paragraph{Further Related Literature} 

\cite{DHP2011} investigate capacities and robust function spaces based on sublinear expectations, in particular, $G$-expectation.
\cite{LMS2022} approach model uncertainty from a reverse perspective, aiming to understand the conditions that a probabilistic model must satisfy to obtain robust analogs of useful properties known in dominated frameworks.
In \cite{BK2012}, risk measures under model uncertainty are studied with a focus on dual representations. Robust duality has further been explored by \cite{BD2018}, who examine general equilibrium theory under Knightian uncertainty.

\paragraph{Outline} The paper is organized as follows. In Section~\ref{sec:prelim:not}, we introduce some notation and present a first discussion of $\mathcal{P}$-sensitivity for both sets and functions. Section~\ref{sec:FL} studies the relation between $\mathcal{P}$-sensitivity and functional localization. Moreover, we define and analyze the different canonical localizations mentioned above. Lastly, in Section~\ref{sec:appl}, we consider applications to optimization problems (see Section~\ref{sec:opt:prob}), convex risk measures (see Section~\ref{sec:rm}), and the superhedging functional (see Section~\ref{sec:FTAP}).

\section{Preliminaries and Notation} \label{sec:prelim:not}

\subsection{Basics}

Throughout this paper, $(\Omega, \mathcal{F})$ is an arbitrary measurable space. We denote by $ca$ the real vector space of all countably additive finite variation set functions $\mu \colon \mathcal{F} \rightarrow \mathbb{R}$, and by $ca_+$ its positive elements ($\mu\in ca_+ \Leftrightarrow \forall A \in \mathcal{F} \ \mu(A) \geq 0$), that is, all finite measures on $(\Omega, \mathcal{F})$. 
\begin{equation}\label{eq:tv}
    \lvert \mu \rvert(A) := \sup \{\mu(B) - \mu(A \setminus B) \mid B \in \mathcal{F}, B \subseteq A\}, \quad  A\in \mathcal F,
\end{equation}
denotes the total variation of $\mu\in ca$, which is a measure $|\mu|\in ca_+$. Given non-empty subsets $\mathfrak{G}$ and $\mathfrak{I}$ of $ca_+$, we say that $\mathfrak{I}$ dominates $\mathfrak{G}$ ($\mathfrak{G} \ll \mathfrak{I}$) if for all $N \in \mathcal{F}$ satisfying $\sup_{\nu \in \mathfrak{I}}  \nu(N) = 0$, we have $\sup_{\mu \in \mathfrak{G}}  \mu(N) = 0$. $\mathfrak{G}$ and $\mathfrak{I}$ are equivalent ($\mathfrak{G} \approx \mathfrak{I}$) if $\mathfrak{G} \ll \mathfrak{I}$ and $\mathfrak{I} \ll \mathfrak{G}$. For the sake of brevity, for $\mu \in ca_+$ we shall write $\mathfrak{G} \ll \mu$, $\mu \ll \mathfrak{I}$, and $\mu \approx \mathfrak{G}$ instead of $\mathfrak{G} \ll \{\mu\}$, $\{\mu\} \ll \mathfrak{I}$, and $\{\mu\} \approx \mathfrak{G}$, respectively.

$\mathfrak{P}(\Omega) \subseteq ca_+$ denotes the set of probability measures on $(\Omega, \mathcal{F})$ and the letters $\mathcal{P}$ and $\mathcal{Q}$ are used to denote non-empty subsets of $\mathfrak{P}(\Omega)$. 
Fix such a set $\mathcal{P}$. We then write $c$ for the induced upper probability $c \colon \mathcal{F} \rightarrow [0, 1]$ defined by
\begin{equation}\label{eq:c}
    c(A) := \sup_{P \in \mathcal{P}} P(A), \quad A \in \mathcal{F}.
\end{equation}
An event $A \in \mathcal{F}$ is called $\mathcal{P}$-polar if $c(A) = 0$. A property holds $\mathcal{P}$-quasi surely (q.s.)\ if it holds outside a $\mathcal{P}$-polar event. We set $ca_{c} := \{\mu\in ca \mid |\mu| \ll \mathcal P\}$, $ca_{c+} :=ca_{+}\cap ca_{c}$, and $\mathfrak{P}_c(\Omega):=\mathfrak{P}(\Omega)\cap ca_{c}$. 

Consider the $\R$-vector space $\mathcal{L}^{0} := \mathcal{L}^{0}(\Omega, \mathcal{F})$ of all real-valued random variables $x \colon \Omega \rightarrow \R$ as well as its subspace $\mathcal{N} := \{x \in \mathcal{L}^{0} \mid c(\lvert x \rvert > 0) = 0\}$. The quotient space $L^{0}_{c} := \mathcal{L}^{0}/\mathcal{N}$ consists of equivalence classes $X$ of random variables up to $\mathcal{P}$-q.s.\ equality comprising representatives $x \in X$. The equivalence class induced by $x \in \mathcal{L}^0$ in $L^{0}_{c}$ is denoted by $[x]_c$. The space $L^{0}_{c}$ carries the so-called $\mathcal{P}$-quasi-sure order $\preccurlyeq_{\mathcal P}$ as a natural vector space order: $X, Y \in L^{0}_{c}$ satisfy $X \preccurlyeq_{\mathcal P} Y$ if for $x \in X$ and $y \in Y$, $x \leq y$ $\mathcal{P}$-q.s., that is, $\{x > y\}$ is $\mathcal P$-polar. In order to facilitate the notation, we suppress the dependence of $\preccurlyeq_{\mathcal P}$ on $\mathcal{P}$ and simply write $\preccurlyeq$ if there is no risk of confusion.

For an event $A \in \mathcal{F}$, $\chi_{A}$ denotes the indicator of the event (i.e., $\chi_{A}(\omega) = 1$ if and only if $\omega \in A$, and $\chi_{A}(\omega) = 0$ otherwise), while $\mathbf{1}_{A} := [\chi_{A}]_{c}$ denotes the generated equivalence class in $L^{0}_{c}$. Throughout the paper, for convenience, we identify the constants $m\in \R$ with the (equivalence classes of) constant random variables they induce. In particular, $m=[m]_c=m\cdot \mathbf{1}_\Omega$.

A subspace of $L^{0}_{c}$ which will turn out to be important for our studies is the space $L^{\infty}_{c}$ of equivalence classes of $\mathcal{P}$-q.s.\ bounded random variables, i.e.,
\begin{equation*}
    L^{\infty}_{c} := \{X \in L^{0}_{c} \mid \exists m > 0 \colon \lvert X \rvert \preccurlyeq m\}.
\end{equation*}
$L^{\infty}_{c}$ is a Banach lattice when endowed with the norm
\begin{equation*}
    \lVert X \rVert_{c,\infty} := \inf \{ m > 0 \mid \lvert X \rvert \preccurlyeq m\}, \quad X \in L^{0}_{c}.
\end{equation*}
$L^{0}_{c+}$ and $L^{\infty}_{c+}$ denote the positive cones of $L^{0}_{c}$ and $L^{\infty}_{c}$, respectively. If $\mathcal{P} = \{P\}$ is given by a singleton and thus $c=P$, we write $L^{0}_{P}$, $L^{\infty}_{P}$, and $[x]_P$ instead of $L^{0}_c$, $L^{\infty}_{c}$, and $[x]_{c}$, and similarly for other expressions in which $c$ appears. Also, the $\mathcal P$-q.s.\ order in this case coincides with the $P$-almost-sure (a.s.)\ order, which we will denote by $\leq_P$ when working with both the $\mathcal P$-q.s.\ order $\preccurlyeq$ for some set $\mathcal P\subseteq \mathfrak{P}(\Omega)$ and the $P$-a.s.\ order for some $P\in \mathfrak{P}_c(\Omega)$. In that respect, note that $\leq_P$ is well defined on $L^0_c$ for any $P\in \mathfrak{P}_c(\Omega)$.

Often we will, as is common practice, identify equivalence classes of random variables with their representatives. However, sometimes it will be helpful to distinguish between them to avoid confusion. Let us clarify this further: For any $X\in L^0_c$, we have $X = \{x \in \mathcal{L}^0 \mid x \sim y\}$ for some $y \in \mathcal{L}^0$ where we write $x \sim y$ to indicate that $\{x\neq y\}$ is $\mathcal{P}$-polar. Any measure $P \in \mathfrak{P}_c(\Omega)$ is consistent with the equivalence relation $\sim$ in the sense that
\begin{equation*}
    \forall x, y\in \mathcal{L}^0 \colon \quad x \sim y \Longrightarrow P(x = y) = 1.
\end{equation*}
In that case we write, for instance, $E_P[X]$ for the expectation of $X$ under $P$, which actually means $E_P[x]$ for any $x \in X$ provided the latter integral is well-defined. Also, we  write expressions like $P(X = Y)$, where $Y\in L^0_c$, actually meaning $P(x = y)$ for arbitrary $x \in X$ and $y \in Y$.

\subsection{\mathinhead{\mathcal{P}}{mathcal{P}}-sensitive Sets and Functions}

Let $\mathcal{P} \subseteq \mathfrak{P}(\Omega)$ and denote by $c$ the corresponding upper probability \eqref{eq:c}. For any given $Q\in \mathfrak{P}_c(\Omega)$, the following map identifies any $X, Y \in L^0_c$ which appear to coincide under $Q$, that is, $Q(x = y) = 1$ for all $x \in X$ and $y \in Y$:
\begin{equation*}
    j_{Q} \colon L^{0}_{c} \rightarrow L^{0}_{Q}, \quad [x]_c \mapsto [x]_Q.
\end{equation*}
Note that $j_Q$ is linear and monotone, in the sense that $X\preccurlyeq Y$ implies $j_Q(X)\leq_Q j_Q(Y)$, and that $j_Q(m)=m$ for all $m\in \R$.

In the sequel, let $\mathcal{X}\subseteq L^0_c$ be non-empty. 

\begin{definition}\label{def:sensitivity}
    A set $\mathcal{C} \subseteq \mathcal{X}$ is called $\mathcal{P}$-sensitive (in $\mathcal{X}$) if either $\mathcal{C}=\emptyset$ or if for all $X\in \mathcal{X}$
    \begin{equation*}
        X\in \mathcal{C} \quad \Leftrightarrow \quad \forall Q\in \mathfrak{P}_c(\Omega)\colon \, j_Q(X)\in j_Q(\mathcal{C}).
    \end{equation*}
\end{definition}

Clearly, the non-trivial implication in Definition~\ref{def:sensitivity} is that $\forall Q\in \mathfrak{P}_c(\Omega)\colon \, j_Q(X)\in j_Q(\mathcal{C})$ implies $X\in \mathcal{C}$. Indeed, a set $\mathcal{C}$ is $\mathcal{P}$-sensitive if it is completely determined by its image under each model $Q \in \mathfrak{P}_c(\Omega)$. Thus, if $X\in \mathcal{X}$ looks like a member of $\mathcal{C}$ under each $Q\in \mathfrak{P}_c(\Omega)$, i.e., $j_{Q}(X) \in j_{Q}(\mathcal{C})$ for all $Q \in  \mathfrak{P}_c(\Omega)$, then in fact $X \in \mathcal{C}$. Trivially, if $\mathcal{P}=\{P\}$, then every set $\mathcal{C}\subseteq L^0_P$ is $P$-sensitive. It has been noticed earlier in, e.g., \cite{LS2025} and \cite{MMS2018} that $\mathcal{P}$-sensitivity of a set $\mathcal{C}$  is an essential property to handle sets in non-dominated robust models.

Before we extend the definition of $\mathcal{P}$-sensitivity from sets to functions, let us collect some notions and useful results for $\mathcal{P}$-sensitive sets. Firstly, sometimes it is useful to work with a stronger representation of $\mathcal{C}$.
 
\begin{definition}
    Let $\mathcal{C} \subseteq \mathcal{X}$. $\mathcal{Q} \subseteq \mathfrak{P}_c(\Omega)$ is called a reduction set for $\mathcal{C}$ if $\mathcal{Q}\neq \emptyset$ and for all $X\in \mathcal{X}$
    \begin{equation} \label{eq:reduction}
        X\in \mathcal{C} \quad \Leftrightarrow \quad \forall Q\in \mathcal{Q}\colon \, j_Q(X)\in j_Q(\mathcal{C}).
    \end{equation}
\end{definition}
 
Clearly, any $\mathcal{P}$-sensitive set admits the reduction set $\mathfrak{P}_c(\Omega)$. The following lemma relates reduction sets to each other and, in particular, shows that any set satisfying \eqref{eq:reduction} is indeed $\mathcal{P}$-sensitive. Its straightforward proof is provided in \cite{LS2025}. 

\begin{lemma}[{\cite[Lemma~2.10]{LS2025}}] \label{lem:set:pSensi:redSet}
    Let $\mathcal{C}\subseteq \mathcal{X}$. 
    \begin{enumerate}
        \item[(i)] Consider a reduction set $\mathcal{Q}_1$ for $\mathcal{C}$ and any other set of probability measures $\mathcal{Q}_{2} \subseteq \mathfrak{P}_c(\Omega)$ such that $\mathcal{Q}_{1} \subseteq \mathcal{Q}_{2}$. Then $\mathcal{Q}_{2}$ is also a reduction set for $\mathcal C$. 
        
        \item[(ii)] $\mathcal C$ is $\mathcal{P}$-sensitive if and only if there exists a reduction set for $\mathcal{C}$.
    \end{enumerate}
\end{lemma}

The reason for considering other reduction sets than simply $\mathfrak{P}_c(\Omega)$ will become evident throughout the paper.

\begin{lemma}[{\cite[Lemma~2.11]{LS2025}}] \label{lem:inf:set}
    Let $I$ be a non-empty index set and let $\mathcal{C}_{\alpha} \subseteq \mathcal{X}$, $\alpha\in I$, be $\mathcal{P}$-sensitive. Then
    \begin{equation*}
        \mathcal{C} := \bigcap_{\alpha \in I} \mathcal{C}_{\alpha}
    \end{equation*}
    is also $\mathcal{P}$-sensitive. If for each $\alpha\in I$, $\mathcal{Q}_\alpha \subseteq \mathfrak{P}_c(\Omega)$ is a reduction set for $\mathcal{C}_{\alpha}$, then $\mathcal Q := \bigcup_{\alpha\in I} \mathcal{Q}_\alpha$ is a reduction set for $\mathcal C$.
\end{lemma}

\begin{lemma}\label{lem:back:forward}
    Let $\mathcal{Q}\subseteq \mathfrak{P}_c(\Omega)$ be non-empty and consider subsets $E_Q\subseteq L^0_Q $, $Q\in \mathcal{Q}$. Define
    \begin{equation*}
        E:=\{X\in \mathcal{X}\mid \forall Q\in \mathcal{Q}\colon j_Q(X)\in E_Q\}.
    \end{equation*}
    Then $E$ is $\mathcal{P}$-sensitive with reduction set $\mathcal{Q}$.
\end{lemma}

\begin{proof}
Suppose that $X\in \mathcal{X}$ is such that $j_Q(X)\in j_Q(E)$ for all $Q\in \mathcal{Q}$. As $j_Q(E)\subseteq E_Q$ for all $Q\in \mathcal{Q}$, we conclude that $X\in E$.
\end{proof}

Note that in the following we write $[-\infty,\infty]$ for the set $\R\cup \{-\infty,\infty\}$ and $(-\infty,\infty]$ for $\R\cup \{\infty\}$.

\begin{definition}
    A function $f \colon \mathcal{X} \to [-\infty,\infty]$ is called \textit{(lower) $\mathcal{P}$-sensitive} if each (lower) level set $E_{r} := \{X \in \mathcal{X} \mid f(X) \leq r\}$, $r\in \R$, is $\mathcal{P}$-sensitive.

    Similarly, $f$ is \textit{upper $\mathcal{P}$-sensitive} if each upper level set $\{X \in \mathcal{X} \mid f(X) \geq r\}$, $r \in \R$, is $\mathcal{P}$-sensitive.
\end{definition}

Note that $f$ is upper $\mathcal{P}$-sensitive if and only if $-f$ is lower $\mathcal{P}$-sensitive. For that reason, we will mainly focus on lower $\mathcal{P}$-sensitive functions in the following and simply refer to them as $\mathcal{P}$-sensitive if there is no risk of confusion.

\begin{definition}
    Let $f \colon \mathcal{X} \to [-\infty,\infty]$. $\mathcal{Q} \subseteq \mathfrak{P}_c(\Omega)$ is called a reduction set for $f$ if $\mathcal{Q}\neq \emptyset$ and $\mathcal{Q}$ is a joint reduction set for all lower level sets $E_r$, $r\in \R$, of $f$.
\end{definition}

Clearly, $\mathfrak{P}_c(\Omega)$ is  a reduction set for every $\mathcal{P}$-sensitive function $f$. In analogy to Lemmas~\ref{lem:set:pSensi:redSet} and~\ref{lem:inf:set}, we obtain the corresponding results for functions:

\begin{lemma} \label{lem:fct:pSensi:redSet}
    Let $f \colon \mathcal{X} \to [-\infty,\infty]$. 
    \begin{enumerate}
        \item[(i)] Consider a reduction set $\mathcal{Q}_1$ for $f$ and any other set $\mathcal{Q}_{2} \subseteq \mathfrak{P}_c(\Omega)$ such that $\mathcal{Q}_{1} \subseteq \mathcal{Q}_{2}$. Then $\mathcal{Q}_{2}$ is a reduction set for $f$ too. 
        
        \item[(ii)] $f$ is $\mathcal{P}$-sensitive if and only if there exists a reduction set for $f$.
    \end{enumerate}
\end{lemma}
\begin{proof}
    Recall Lemma~\ref{lem:set:pSensi:redSet}.
\end{proof}

\begin{lemma}
    Suppose that the members of the family of functions $f_\alpha:\mathcal{X}\to [-\infty,\infty]$ , $\alpha \in I$, for some non-empty index set $I$ are all $\mathcal{P}$-sensitive. Then $f:\mathcal{X}\to [-\infty,\infty], \, X\mapsto \sup_{\alpha\in I} f_\alpha(X)$, is also $\mathcal{P}$-sensitive.
\end{lemma}
\begin{proof}
    Recall Lemma~\ref{lem:inf:set}.
\end{proof}

\begin{definition}
    Let $\mathcal{Q} \subseteq \mathfrak{P}_{c}(\Omega)$ be non-empty and $Q \in \mathfrak{P}_{c}(\Omega)$. A function $f \colon \mathcal{X} \to [-\infty, \infty]$ is called \textit{$Q$-consistent} (on $\mathcal{X}$) if for any $X, Y \in \mathcal{X}$ with $j_{Q}(X) = j_{Q}(Y)$, it holds that $f(X) = f(Y)$.
    
    A family of functions $f^{Q} \colon \mathcal{X} \to [-\infty, \infty]$, $Q \in \mathcal{Q}$, is called \textit{$\mathcal{Q}$-consistent} (on $\mathcal{X}$) if $f^Q$ is $Q$-consistent for all $Q \in \mathcal{Q}$.
\end{definition}

\begin{remark}\label{rem:local} Let $Q \in \mathfrak{P}_{c}(\Omega)$.
    Any function $g \colon j_{Q}(\mathcal{X}) \to [-\infty, \infty]$ defines a canonical $Q$-consistent function $f\colon \mathcal{X} \to [-\infty, \infty]$ by $f(X) := g(j_{Q}(X))$, $X \in \mathcal{X}$. Conversely, any $Q$-consistent function $f\colon \mathcal{X} \to [-\infty, \infty]$ defines a function $g \colon j_{Q}(\mathcal{X}) \to [-\infty, \infty]$ by $g(X)=f(Y)$ for arbitrary $Y\in j_Q^{-1}(X)\cap \mathcal{X}$. We will, with some abuse of notation,  denote this function $g$ by $f\circ j_Q^{-1}$. \hfill $\diamond$
\end{remark}

Again, analogously to Lemma~\ref{lem:back:forward}, we obtain the following results for functions:

\begin{lemma}\label{lem:loc:rep}
    Let $\mathcal{Q} \subseteq \mathfrak{P}_c(\Omega)$ be non-empty and let $(f^{Q})_{Q \in \mathcal{Q}}$ be a $\mathcal{Q}$-consistent family of functions on $\mathcal{X}$. Then the function
    \begin{equation*}
        f(X) := \sup_{Q \in \mathcal{Q}} f^Q(X), \quad X \in \mathcal{X},
    \end{equation*}
    is $\mathcal{P}$-sensitive with reduction set $\mathcal{Q}$.
\end{lemma}
\begin{proof}
    Fix $r \in \R$ and let $E_r:=\{Y\in \mathcal{X}\mid f(Y)\leq r\}$. Consider $X \in \mathcal{X}$ such that $j_{Q}(X) \in j_{Q}(E_{r})$ for all $Q \in \mathcal{Q}$. Then, for any $Q \in \mathcal{Q}$, there exists $X^{Q} \in E_{r}$ such that $j_{Q}(X) = j_{Q}(X^{Q})$. As $(f^{Q})_{Q \in \mathcal{Q}}$ is $\mathcal{Q}$-consistent, it follows that for any $Q \in \mathcal{Q}$
    \begin{equation*}
        f^{Q}(X) = f^{Q}(X^{Q}) \leq f(X^Q) \leq r.
    \end{equation*}
    Consequently,
    \begin{equation*}
        f(X)=\sup_{Q \in \mathcal{Q}} f^Q(X)\leq r,
    \end{equation*}
    that is, $X \in E_{r}$. Hence, $E_{r}$ is $\mathcal{P}$-sensitive with reduction set $\mathcal{Q}$.
\end{proof}

Of course, a set is $\mathcal{P}$-sensitive if and only if its indictor is $\mathcal{P}$-sensitive as a function:

\begin{lemma}
    A set $\mathcal{C}\subseteq \mathcal{X}$ is $\mathcal{P}$-sensitive if and only if its convex analytic indicator function
    \begin{equation*}
        \delta(X\mid \mathcal{C}) := 
        \begin{cases}
            0, &\text{if } X \in \mathcal{C}, \\
            \infty, &\text{if } X \not\in \mathcal{C},
        \end{cases}
        \quad X\in \mathcal{X},
    \end{equation*}
    is $\mathcal{P}$-sensitive.
\end{lemma}

The following proposition shows that $\mathcal{P}$-sensitivity is automatically satisfied for a large class of lower semicontinuous quasi-convex functions under mild assumptions on the dual space.

\begin{proposition}
    Let  $\mathcal{X}\subseteq L^0_c$ and $\mathcal{Y} \subseteq \mathrm{ca}_{c}$ be subspaces such that $\langle \mathcal{X},\mathcal{Y}\rangle$ is a dual pair, and $\mathcal{Y}$ satisfies $\mu \in \mathcal{Y} \, \Rightarrow |\mu|\in \mathcal{Y}$. Then each $\sigma(\mathcal{X}, \mathcal{Y})$-lower semicontinuous quasi-convex function $f \colon \mathcal{X} \to [-\infty,\infty]$ is $\mathcal{P}$-sensitive.
\end{proposition}
\begin{proof}
    If $f$ is quasi-convex and lower semicontinuous, then each level set $E_{r}$ is convex and $\sigma(\mathcal{X}, \mathcal{Y})$-closed. According to \cite[Theorem 5.13]{LS2025}, $E_{r}$ is $\mathcal{P}$-sensitive. 
\end{proof}

Regarding sufficient properties to ensure that a set is $\mathcal{P}$-sensitive, we refer to \cite{LS2025}. In the case of $\mathcal{P}$-sensitive functions, sufficient conditions will be provided throughout the remainder of this paper. The question we address next is whether there are robust models $(\Omega, \mathcal{F}, \mathcal{P})$ such that all sets $\mathcal{C} \subseteq \mathcal{X}$, and thus all functions $f:\mathcal{X} \to [-\infty, \infty]$, are $\mathcal{P}$-sensitive:

\begin{theorem}\label{thm:all:Psensi}
    Suppose that $\{\mathbf{1}_A\mid A\in \mathcal{F}\}\subseteq \mathcal{X}$. All subsets of $\mathcal{X}$ are $\mathcal{P}$-sensitive if and only if $\mathcal{P}$ is dominated, that is, there is $Q\in \mathfrak{P}(\Omega)$ such that $\mathcal{P}\ll Q$. 
\end{theorem}

The condition $\{\mathbf{1}_A\mid A\in \mathcal{F}\}\subseteq \mathcal{X}$ is in particular satisfied if $L^\infty_c\subseteq \mathcal{X}$. 
The proof of Theorem~\ref{thm:all:Psensi} needs some preparation. As already observed in \cite{LS2025}, $\mathcal{P}$-sensitivity is closely related to the problem of aggregation of families of (equivalence classes of) random variables as introduced in the following. 

\begin{definition}
    Let $\mathcal{Q}\subseteq \mathfrak{P}_c(\Omega)$ be non-empty. 
    \begin{enumerate}
        \item A family $(X^Q)_{Q\in \mathcal{Q}}\subseteq \mathcal{X}$ is called \textit{$\mathcal{Q}$-coherent} (in $\mathcal{X}$) if there exists $X\in \mathcal{X}$ such that $j_Q(X^Q)=j_Q(X)$ for all $Q\in \mathcal{Q}$. In that case, $X$ is called a \textit{$\mathcal{Q}$-aggregator} (in $\mathcal{X}$) of the family $(X^Q)_{Q\in \mathcal{Q}}$.
        \item A set $\mathcal{C}\subseteq \mathcal{X}$ is \textit{$\mathcal{Q}$-stable} (in $\mathcal{X}$) if $\mathcal{C}$ contains all $\mathcal{Q}$-aggregators (in $\mathcal{X}$) of all $\mathcal{Q}$-coherent families $(X^Q)_{Q\in \mathcal{Q}}\subseteq \mathcal{C}$.
        \item Let $(X^Q)_{Q\in \mathcal{Q}}$ be $\mathcal{Q}$-coherent in $\mathcal{X}$. A $\mathcal{Q}$-aggregator $X$ of $(X^Q)_{Q\in \mathcal{Q}}$ in $\mathcal{X}$ is called \textit{trivial} if there exists $Q\in \mathcal{Q}$ such that $X = X^Q$. $X$ is called \textit{non-trivial} if $X \neq X^Q$ for all $Q\in \mathcal{Q}$ (i.e., for all $Q\in \mathcal{Q}$ there exists $P\in \mathcal{P}$ such that $P(X \neq X^Q) > 0$). 
    \end{enumerate}
\end{definition}

The following is easily verified, see \cite[Proposition 5.8]{LS2025} for a proof in case $\mathcal{X}=L^0_c$.

\begin{lemma}\label{lem:Qstable}
    A set $\mathcal{C}\subseteq \mathcal{X}$ is $\mathcal{P}$-sensitive with reduction set $\mathcal{Q}$ if any only if $\mathcal{C}$ is $\mathcal{Q}$-stable.
\end{lemma}

In particular, $\mathcal{C}$ is $\mathcal{P}$-sensitive if and only if  $\mathcal{C}$ is $\mathfrak{P}_c(\Omega)$-stable. 

\begin{lemma}\label{lem:all:sets:Psensi:1}
    Every set $\mathcal{C}\subseteq \mathcal{X}$ is $\mathcal{P}$-sensitive if and only if every $\mathfrak{P}_c(\Omega)$-aggregator (of any $\mathfrak{P}_c(\Omega)$-coherent family) is trivial.
\end{lemma}
\begin{proof}
    Suppose that every set $\mathcal{C}\subseteq \mathcal{X}$ is $\mathcal{P}$-sensitive, and let $(X^Q)_{Q\in \mathcal{Q}}\subseteq L^0_c$ be $\mathfrak{P}_c(\Omega)$-coherent. By assumption, $\mathcal{C}:=\{X^Q\mid Q\in \mathfrak{P}_c(\Omega)\}$ is $\mathcal{P}$-sensitive and thus $\mathfrak{P}_c(\Omega)$-stable. Hence, any $\mathfrak{P}_c(\Omega)$-aggregator $X\in \mathcal{X}$ of $(X^Q)_{Q\in \mathcal{Q}}$ must satisfy $X\in \mathcal{C}$.
    
    Conversely, suppose that every $\mathfrak{P}_c(\Omega)$-aggregator is trivial. Consider an arbitrary non-empty set $\mathcal{C}\subseteq \mathcal{X}$ and any $\mathfrak{P}_c(\Omega)$-coherent family $(X^Q)_{Q\in \mathcal{Q}}\subseteq \mathcal{C}$ with $\mathfrak{P}_c(\Omega)$-aggregator $X\in \mathcal{X}$. Then $X\in \mathcal{C}$, because $X=X^Q\in \mathcal{C}$ for some $Q\in \mathfrak{P}_c(\Omega)$. Hence, $\mathcal{C}$ is $\mathfrak{P}_c(\Omega)$-stable and thus $\mathcal{P}$-sensitive. 
\end{proof}

Lemma~\ref{lem:all:sets:Psensi:1} can be further refined:

\begin{lemma}\label{lem:all:sets:Psensi:2} 
    Suppose that $\{\mathbf{1}_A\mid A\in \mathcal{F}\}\subseteq \mathcal{X}$. All sets $\mathcal{C}\subseteq \mathcal{X}$ are $\mathcal{P}$-sensitive if and only if $1$ is a trivial $\mathfrak{P}_c(\Omega)$-aggregator of every family of indicators $(\mathbf{1}_{A^Q})_{Q\in \mathfrak{P}_c(\Omega)}$ such that $Q(A^Q)=1$ for all $Q\in \mathfrak{P}_c(\Omega)$.
\end{lemma}

\begin{proof}
    First note that any family of indicators $(\mathbf{1}_{A^Q})_{Q\in \mathfrak{P}_c(\Omega)}$ such that $Q(A^Q)=1$ for all $Q\in \mathfrak{P}_c(\Omega)$ is $\mathfrak{P}_c(\Omega)$-coherent with $\mathfrak{P}_c(\Omega)$-aggregator $1$. Hence, if every set $\mathcal{C}\subseteq L^0_c$ is $\mathcal{P}$-sensitive, then this aggregator must be trivial by Lemma~\ref{lem:all:sets:Psensi:1}.
    
    Conversely, consider a non-empty set $\mathcal{C}\subseteq \mathcal{X}$ and any $\mathfrak{P}_c(\Omega)$-coherent family $(X^Q)_{Q\in \mathfrak{P}_c(\Omega)}\subseteq \mathcal{C}$ with $\mathfrak{P}_c(\Omega)$-aggregator $X\in \mathcal{X}$. Let $x \in X$ and $x^{Q} \in X^{Q}$, and set $A^Q := \{x^Q = x\}$, $Q\in \mathfrak{P}_c(\Omega)$. Then $Q(A^Q)=1$ for all $Q\in \mathfrak{P}_c(\Omega)$. As $1$ is a trivial $\mathfrak{P}_c(\Omega)$-aggregator of $(\mathbf{1}_{A^Q})_{Q\in \mathfrak{P}_c(\Omega)}$, there is $Q\in \mathfrak{P}_c(\Omega)$ such that $1=\mathbf{1}_{A^Q}$. This implies that for all $P\in \mathcal{P}$ we have $1=P(1=\mathbf{1}_{A^Q})=P(X^Q=X)$. Hence, $X^Q=X$, and thus $X\in \mathcal{C}$. Therefore, $\mathcal{C}$ is $\mathfrak{P}_c(\Omega)$-stable.
\end{proof}

\begin{proof}[Proof of Theorem~\ref{thm:all:Psensi}]
    Suppose that $\mathcal{P}\ll Q$. For all $P\in \mathcal{P}$, denote by $dP/dQ\in \mathcal{L}^1(\Omega,\mathcal{F}, Q)$ a version of the Radon-Nikodym density of $P$ with respect to $Q$. Let $z$ be a version of 
    \begin{equation*}
        \operatorname{esssup}\bigg\{\chi_{\{\frac{dP}{dQ}>0\}}\biggm\vert P\in \mathcal{P}\bigg\}\in L^0(\Omega,\mathcal{F},Q).
    \end{equation*}
    Then $A := \{z > 0\}\in \mathcal{F}$, and one verifies that the conditional probability measure $\widetilde Q:=Q(\, \cdot\mid A)$ satisfies $\widetilde Q\approx \mathcal{P}$. In particular, $\widetilde Q\in \mathfrak{P}_c(\Omega)$ and indeed $L^0_c=L^0(\Omega,\mathcal{F}, \widetilde Q)$. Any set $\mathcal{C}\subseteq \mathcal{X}\subseteq L^0_c$ is thus $\mathcal{P}$-sensitive with reduction set $\{\widetilde Q\}$.
    
    Conversely, assume that all sets $\mathcal{C}\subseteq \mathcal{X}$ are $\mathcal{P}$-sensitive. Then, by Lemma~\ref{lem:all:sets:Psensi:2}, there exists $R\in \mathfrak{P}_c(\Omega)$ such that $\mathbf{1}_A=1$ for all $A\in \mathcal{F}$ with $R(A)=1$. Indeed, otherwise for all $\widetilde{R} \in \mathfrak{P}_{c}(\Omega)$ there would exist $A \in \mathcal{F}$ with $\widetilde{R}(A) = 1$ such that $\mathbf{1}_{A} \neq 1$. Then $1$ is a non-trivial $\mathfrak{P}_{c}(\Omega)$-aggregator, contradicting Lemma~\ref{lem:all:sets:Psensi:2}. We claim that $\mathcal{P}\ll R$. To this end, suppose that $N\in \mathcal{F}$ satisfies $R(N)=0$. Then $R(\Omega\setminus N)=1$ and therefore $\mathbf{1}_{\Omega\setminus N}=1$, that is, $1=P(\mathbf{1}_{\Omega\setminus N}=1)=P(\Omega\setminus N)$ for all $P\in \mathcal{P}$. Consequently, $P(N)=0$ for all $P\in \mathcal{P}$.
\end{proof}

\begin{example}
    Let $(\Omega, \mathcal{F}) = ([0, 1], \mathcal{B}([0, 1]))$, where $\mathcal{B}([0, 1])$ denotes the Borel-$\sigma$-algebra over $[0, 1]$. Further, let $\mathcal{P} := \{\delta_{\omega} \mid \omega \in [0, 1]\}$ be the set of all Dirac measures. Note that in this case $\preccurlyeq$ coincides with the pointwise order and that $L^0_c=\mathcal{L}^0(\Omega,\mathcal{F})$ is simply the space of all random variables. Moreover, $ca_c=ca$ and, in particular, $\mathfrak{P}_c(\Omega)=\mathfrak{P}(\Omega)$. For any $Q\in \mathfrak{P}(\Omega)$ denote by $A(Q):=\{\omega \in \Omega\mid Q(\{\omega\})>0 \}$ the set of $Q$-atoms, and recall that the set $A(Q)$ is at most countable. Therefore, $\Omega\setminus A(Q)\neq \emptyset$. Since any probability measure $Q$ dominating $\mathcal{P}$ would have to satisfy $Q(\{\omega\})>0$ for all $\omega\in \Omega$, that is, $A(Q)=\Omega$, which is not possible, $\mathcal{P}$ is not dominated.
    
    Moreover, for instance, by removing an arbitrary $\omega \in \Omega\setminus A(Q)$ from $\Omega$, it follows that, for any $Q\in \mathfrak{P}(\Omega)$, we may choose $B^Q\in \mathcal{F}$ such that $Q(B^Q)=1$ but $B^Q\neq \Omega$, and thus $1\neq \mathbf{1}_{B^Q}$. The family $(\mathbf{1}_{B^Q})_{Q\in \mathfrak{P}(\Omega)}$ is $\mathfrak{P}(\Omega)$-coherent with non-trivial $\mathfrak{P}(\Omega)$-aggregator $1$ in any $\mathcal{X}\subseteq L^0_c$ such that $\{\mathbf{1}_A\mid A\in \mathcal{F}\}\subseteq \mathcal{X}$. The set
    \begin{equation*}
        \mathcal{C} := \{\mathbf{1}_{B^Q}\mid Q\in \mathfrak{P}(\Omega)\}
    \end{equation*}
    is not $\mathcal{P}$-sensitive in any $\mathcal{X}\subseteq L^0_c$ such that $\{\mathbf{1}_A\mid A\in \mathcal{F}\}\subseteq \mathcal{X}$, because it is not $\mathfrak{P}(\Omega)$-stable. This can also be directly verified without invoking aggregation and stability, since for all $Q\in \mathfrak{P}(\Omega)$, we have $j_Q(1)=j_Q(\mathbf{1}_{B^Q})\in j_Q(\mathcal{C})$. However, $1\notin \mathcal{C}$, since $1\neq \mathbf{1}_{B^Q}$ for all $Q\in \mathfrak{P}(\Omega)$. \hfill $\diamond$
\end{example}

\begin{remark}
    Whether a set is $\mathcal{P}$-sensitive not only depends on $(\Omega,\mathcal{F}, \mathcal{P})$ but also on the domain $\mathcal{X}$. Clearly, if $\mathcal{C}\subseteq \mathcal{X}\subseteq \widetilde{\mathcal{X}}$ is  $\mathcal{P}$-sensitive in $\widetilde{\mathcal{X}}$, then $\mathcal{C}$ is $\mathcal{P}$-sensitive in $\mathcal{X}$. Indeed, in $\mathcal{X}$ there are less $\mathfrak{P}_c(\Omega)$-coherent families or $\mathfrak{P}_c(\Omega)$-aggregators. \hfill $\diamond$
\end{remark}

\section{\mathinhead{\mathcal{P}}{mathcal{P}}-Sensitivity and Functional Localization}\label{sec:FL}

As before, let $\mathcal{X} \subseteq L^{0}_{c}$ be non-empty and let $f \colon \mathcal{X} \to [-\infty,\infty]$.

\begin{definition}
    Let $f^Q:\mathcal{X} \to [-\infty, \infty]$, $Q\in \mathcal{Q}$, be a $\mathcal{Q}$-consistent family of functions, where $\mathcal{Q} \subseteq \mathfrak{P}_{c}(\Omega)$ is non-empty. If  
    \begin{equation*}
        f(X) = \sup_{Q\in \mathcal{Q}} f^{Q}(X),\quad X \in \mathcal{X},
    \end{equation*}
    then $(f^{Q})_{Q \in \mathcal{Q}}$ is called a (functional) {\em $\mathcal{Q}$-localization} of $f$.
\end{definition}

Note that the aggregator of the functions $f^Q$ in the definition of a $\mathcal{Q}$-localization is the supremum, which may be seen as a worst-case approach when the $f^Q$ represent some type of local risk assessment under the model $Q \in \mathcal{Q}$.

By Lemma~\ref{lem:loc:rep}, if $f$ admits a $\mathcal{Q}$-localization, then $f$ is necessarily $\mathcal{P}$-sensitive. We will show that the converse is also true. To this end, we consider a particular family of local functions which will turn out to be a localization of $f$ whenever $f$ is $\mathcal{P}$-sensitive. For $Q \in \mathfrak{P}_{c}(\Omega)$, let $f^{Q}_{E} \colon \mathcal{X} \to \R$ be given by
\begin{align*}
    \begin{split}
        f^{Q}_{E}(X) &:= \inf \{r \in \R \mid j_{Q}(X) \in j_{Q}(E_{r})\} \quad (\inf \emptyset:=\infty),
    \end{split}
\end{align*}
where, as before, $E_{r} := \{X \in \mathcal{X} \mid f(X) \leq r\}$, $r\in \R$, denote the level sets of $f$. Moreover, for non-empty  $\mathcal{Q} \subseteq \mathfrak{P}_{c}(\Omega)$, we define 
\begin{equation*}
    f^{\mathcal Q}_{E}(X) := \sup_{Q \in \mathcal{Q}} f^{Q}_{E}(X), \quad X \in \mathcal{X}, 
\end{equation*}
and $f_E := f^{\mathfrak{P}_{c}(\Omega)}_{E}$. 

\begin{theorem}\label{thm:Psensi:localization} 
    There exists a $\mathcal{Q}$-localization of $f$ if and only if $f$ is $\mathcal{P}$-sensitive with reduction set $\mathcal{Q}$. In that case $f= f^{\mathcal{Q}}_{E}=f^{\mathcal{Q}'}_{E}$ for every $\mathcal{Q}'$ satisfying $\mathcal{Q}\subseteq \mathcal{Q}'\subseteq \mathfrak{P}_c(\Omega)$, i.e., $(f^Q_E)_{Q\in \mathcal{Q}'}$ is a $\mathcal{Q}'$-localization of $f$.
\end{theorem}

In particular, $f=f_E$ if and only if $f$ is $\mathcal{P}$-sensitive, because $\mathfrak{P}_{c}(\Omega)$ is a reduction set for every  $\mathcal{P}$-sensitive function.
The proof of Theorem~\ref{thm:Psensi:localization} is postponed to after the following lemma, the proof of which is straightforward and left to the reader.

\begin{lemma} \label{lem:aux:1}
    Let $f:\mathcal{X}\to [-\infty,\infty]$ and denote by $E_r$, $r\in \R$, the level sets of $f$. Moreover, let $Q\in \mathfrak{P}_c(\Omega)$. 
    \begin{enumerate}
        \item[(i)] $f^Q_E$ is $Q$-consistent.
        
        \item[(ii)]  $f^{Q}_{E} \leq f$.

        \item[(iii)] $f^Q_E(X) = \inf\{f(Y) \mid Y \in \mathcal{X} \colon j_{Q}(X) = j_{Q}(Y)\}$, $X \in \mathcal{X}$.

        \item[(iv)] If $f^Q:\mathcal{X}\to [-\infty,\infty]$ is $Q$-consistent such that $f^{Q} \leq f$, then $f^{Q} \leq f^{Q}_{E}$.
        
        \item[(v)] Let $\mathcal{Q}\subseteq \mathcal{Q}'\subseteq \mathfrak{P}_c(\Omega)$, then  $f^{\mathcal{Q}}_{E}\leq f^{\mathcal{Q}'}_{E}$. In particular, $f^{\mathcal{Q}}_{E}\leq f_E\leq f$.
    \end{enumerate}
\end{lemma}

\begin{proof}[Proof of Theorem~\ref{thm:Psensi:localization}]
    If $f$ admits a $\mathcal{Q}$-localization, then $f$ is $\mathcal{P}$-sensitive with reduction set $\mathcal{Q}$ by Lemma~\ref{lem:loc:rep}.
    
    Conversely, let $f$ be $\mathcal{P}$-sensitive with reduction set $\mathcal{Q}$. We show that $f=f^{\mathcal{Q}}_{E}$. In view of Lemma~\ref{lem:aux:1}, this is equivalent to $f\leq f^{\mathcal{Q}}_{E}$ since always $f\geq f^{\mathcal{Q}}_{E}$. Consider $X\in \mathcal{X}$ such that $f^{\mathcal{Q}}_{E}(X)<\infty$ and let $r\in \R$ such that $f^{\mathcal{Q}}_{E}(X) < r$. Then, for all $Q \in \mathcal{Q}$, $f^{Q}_{E}(X) \leq f^{\mathcal{Q}}_{E}(X) < r$ and thus $j_{Q}(X) \in j_{Q}(E_{r})$. As $\mathcal{Q}$ is a reduction set for $E_{r}$, we obtain that $X \in E_{r}$ and thus, $f(X)\leq r$. It follows that 
    \begin{equation*}
        f^{\mathcal{Q}}_{E}(X) = \inf\{r \in \R \mid f^{\mathcal{Q}}_{E}(X) < r\} \geq \inf\{r \in \R \mid f(X)\leq r \} =f(X).
    \end{equation*}
  Finally, recall that if $\mathcal{Q}$ is a reduction set for $f$, then so is any $\mathcal{Q}'$ with $\mathcal{Q} \subseteq \mathcal{Q}' \subseteq \mathfrak{P}_c(\Omega)$ (see Lemma~\ref{lem:fct:pSensi:redSet}). 
\end{proof}

\begin{remark} 
    Let $\mathcal{Q} \subseteq \mathfrak{P}_{c}(\Omega)$ be non-empty. Theorem~\ref{thm:Psensi:localization} implies that a function $f \colon \mathcal{X} \to [-\infty, \infty]$ admits a representation
    \begin{equation} \label{eq:inf:loc}
        f(X) = \inf_{Q \in \mathcal{Q}} f^{Q}(X), \quad X \in \mathcal{X},
    \end{equation}
    where $(f^{Q})_{Q \in \mathcal{Q}}$ is a $\mathcal{Q}$-consistent family of functions, if and only if $f$ is upper $\mathcal{P}$-sensitive and $\mathcal{Q}$ is a joint reduction set for every upper level set of $f$.  Moreover, \eqref{eq:inf:loc} corresponds to a worst-case approach when the $f^Q$ represent a local utility assessment under the model $Q \in \mathcal{Q}$. \hfill $\diamond$
\end{remark}

\subsection{Properties of \mathinhead{\mathcal{P}}{mathcal{P}}-Sensitive Functions}\label{sec:properties}

In this section, we analyze the relationship between properties of $f$ and its localizations. We now assume that $\mathcal{X}$ is a linear space containing the constants.  We will consider the following properties known, e.g., from the study of risk measures:
\begin{enumerate}
    \item[(i)] Monotonicity: For all $X_{1}, X_{2} \in \mathcal{X}$ such that $X_{1} \preccurlyeq X_{2}$, we have $f(X_{1}) \leq f(X_{2})$.
    
    \item[(ii)] Cash-additivity: For all $X \in \mathcal{X}$ and $m \in \R$, we have $f(X + m) = f(X) + m$.

     \item[(iii)] Quasi-convexity: For all $X_{1}, X_{2} \in \mathcal{X}$ and $\lambda \in [0, 1]$, we have $f(\lambda X_{1} + (1 - \lambda) X_{2}) \leq \max \{f(X_{1}), f(X_{2})\}$.

    \item[(iv)] Convexity: For all $X_{1}, X_{2} \in \mathcal{X}$ and $\lambda \in [0, 1]$, we have $f(\lambda X_{1} + (1 - \lambda) X_{2}) \leq \lambda f(X_{1}) + (1~-~\lambda) f(X_{2})$.

    \item[(v)] Positive homogeneity: For all $X \in \mathcal{X}$ and $\lambda \geq 0$, we have $f(\lambda X) = \lambda f(X)$.

    \item[(vi)] Subadditivity: For all $X_{1}, X_{2} \in \mathcal{X}$, we have $f(X_{1} + X_{2}) \leq f(X_{1}) + f(X_{2})$.
\end{enumerate}

\begin{lemma}
    Consider a $\mathcal{Q}$-localization $(f^{Q})_{Q \in \mathcal{Q}}$ of $f$.
    \begin{enumerate}
        \item[(i)] $f$ is monotone whenever each $f^{Q}$, $Q\in \mathcal{Q}$, is monotone. 
        
        \item[(ii)] $f$ is cash-additive whenever each $f^{Q}$, $Q \in \mathcal{Q}$, is cash-additive.

        \item[(iii)] $f$ is (quasi-)convex whenever each $f^{Q}$, $Q \in \mathcal{Q}$, is (quasi-)convex.
        
        \item[(iv)] $f$ is positively homogeneous whenever each $f^{Q}$, $Q \in \mathcal{Q}$, is positively homogeneous.

        \item[(v)] $f$ is subadditive whenever each $f^{Q}$, $Q \in \mathcal{Q}$, is subadditive.
    \end{enumerate}
\end{lemma}
\begin{proof}
    Recall that $f(X) = \sup_{Q \in \mathcal{Q}} f^{Q}(X)$. Then the assertions follow. 
\end{proof}

\begin{lemma} \label{lem:f:fQE}
   Let $Q \in \mathfrak{P}_c(\Omega)$.
    \begin{enumerate}
        \item[(i)] If $f$ is monotone, then $f^{Q}_{E}$ is monotone. Moreover, $j_Q(X_1)\leq_Q j_Q(X_2)$ implies $f^{Q}_{E}(X_1)\leq f^{Q}_{E}(X_2)$.
        
        \item[(ii)] If $f$ is cash-additive, then  $f^{Q}_{E}$ is cash-additive.

        \item[(iii)] If $f$ is (quasi-)convex, then $f^{Q}_{E}$ is (quasi-)convex.
        
        \item[(iv)] If $f$ is positively homogeneous, then $f^{Q}_{E}$ satisfies $f^{Q}_{E}(\lambda X)=\lambda f^Q_E(X)$ for all $X\in \mathcal{X}$ and $\lambda>0$. Moreover, either $f^{Q}_{E}(0)=-\infty$ or $f^{Q}_{E}$ is positively homogeneous.

        \item[(v)] If $f$ is subadditive, then $f^{Q}_E$ is subadditive.
    \end{enumerate}
\end{lemma}
\begin{proof}
     (i) Let $X_{1}, X_{2} \in \mathcal{X}$ such that $X_{1} \preccurlyeq X_{2}$. Then $j_Q(X_{1}) \leq_Q j_Q(X_{2})$. Hence, it suffices to show the second assertion. To this end, suppose that $j_Q(X_{1}) \leq_Q j_Q(X_{2})$.  If $f^Q_E(X_{2}) = \infty$, there is nothing to show. Let $r \in \R$ such that $f^Q_E(X_{2}) < r$. There is $\widetilde X_{2} \in \mathcal{X}$ such that $f(\widetilde X_2) \leq r $ and $j_{Q}(\widetilde X_{2}) = j_Q(X_{2})$. Set $A := \{\omega \in \Omega \mid g_{1}(\omega) \leq g_{2}(\omega)\}$ where $g_{1} \in X_{1}$ and $g_{2} \in \widetilde X_{2}$. Then $Q(A)=1$ and 
    \begin{equation*}
         \mathbf{1}_{A} X_{1} + \mathbf{1}_{A^{c}} \widetilde X_{2} \preccurlyeq  \widetilde X_2.
    \end{equation*}
    Since $f$ is monotone, $f(\mathbf{1}_{A} X_{1} + \mathbf{1}_{A^{c}} \widetilde X_{2}) \leq f(\widetilde X_{2}) \leq r$. As $j_{Q}(X_{1}) = j_{Q}(\mathbf{1}_{A} X_{1} + \mathbf{1}_{A^{c}} \widetilde X_{2})$, we conclude by Lemma~\ref{lem:aux:1} that $f_E^Q(X_1)\leq f(\mathbf{1}_{A} X_{1} + \mathbf{1}_{A^{c}} \widetilde X_{2})\leq r$. Since $r > f^{Q}_{E}(X_{2})$ was arbitrary, it follows that $f^{Q}_{E}(X_{1}) \leq f^{Q}_{E}(X_{2})$.

    \smallskip\noindent
    (ii), (iii), and (v) follow from Lemma~\ref{lem:aux:1}~(iii) and linearity of $j_Q$.

    \smallskip\noindent
    (iv) In case $\lambda>0$ and $X\in \mathcal{X}$, Lemma~\ref{lem:aux:1}~(iii) yields the assertion. Since $f(0)=0$, we have that $f^Q_E(0)\leq 0$. Suppose that $f^Q_E(0)< 0$. Then there must be $X\neq 0$ and $r<0$ such that $f(X)\leq r$ and $j_Q(X)=0$. But then also $j_Q(\lambda X)=0$ for all $\lambda>0$. Thus, $f^Q_E(0)\leq f(\lambda X)=\lambda f(X)=- \infty$.
\end{proof}

\subsection{Robust Dual Representation and \mathinhead{\mathcal{P}}{mathcal{P}}-Sensitivity}

In this section, we will show that a function $f:\mathcal{X}\to [-\infty,\infty]$ admits a dual representation over $ca_c$ only if it is $\mathcal{P}$-sensitive. Moreover, in that case we obtain another functional localization of $f$ via a \textit{dual approach}. In contrast, the localization defined by the functions $f^Q_E$, which are based on the level sets, can be viewed as a \textit{primal approach} to represent $f$.

Throughout this section, we assume that $\mathcal{X}\subseteq L^0_c$ is a linear space containing the constants. Let 
\begin{equation*}
    ca_c(\mathcal{X}):= \bigg\{\mu \in ca_c \biggm\vert \int X d\mu \text{ is well-defined and finite for all } X\in \mathcal{X}\bigg\}.
\end{equation*}
Note that always $0\in ca_c(\mathcal{X})$. A function $f \colon \mathcal{X} \to [-\infty, \infty]$ is said to admit a \textit{dual representation} over $ca_c(\mathcal{X})$ if  
\begin{equation}\label{eq:dualf1}
    f(X) = \sup_{\mu \in ca_c(\mathcal{X})} \int X d\mu - f^{*}(\mu), \quad X\in \mathcal{X},
\end{equation}
where $f^{*} \colon ca_c(\mathcal{X}) \to [-\infty,\infty]$. Note that $f^\ast(\mu)=\infty$, means that $\mu$ is effectively not considered in \eqref{eq:dualf1}. Also, $f^\ast$ in \eqref{eq:dualf1} is not necessarily unique, meaning that $f$ may admit a representation of the form \eqref{eq:dualf1} for different  $f^\ast$. In order to deal with the latter issue, we consider, as usual, the minimal $f^\ast$ for which $f$ admits a representation \eqref{eq:dualf1}. This so-called \textit{dual function} $\overline{f}:ca_c(\mathcal{X})\to [-\infty,\infty]$ is given by \begin{equation*}
    \overline f(\mu):=\sup_{X\in \mathcal{X}} \int X d\mu - f(X), \quad \mu \in ca_{c}(\mathcal{X}).
\end{equation*}
Note that $\overline f$ is indeed defined for every function $f:\mathcal{X}\to [-\infty,\infty]$, including those that do not admit a dual representation. If $f$ takes the value $-\infty$, then $\overline{f}\equiv \infty$. The \textit{dual regularization} $f_D$ of $f$ is given by
\begin{equation*}
    f_D(X) := \sup_{\mu\in ca_c(\mathcal{X})} \int X d\mu -\overline f(\mu), \quad X\in \mathcal{X}.
\end{equation*}
Clearly, $f_D$ is convex and $f_D\leq f$. It is well-known that a function $f$ admits a dual representation of the form \eqref{eq:dualf1} if and only if  $f=f_D$. Indeed, suppose that $f$ admits a dual representation \eqref{eq:dualf1}, then by definition of $\overline f$, we have $f^\ast \geq \overline f$. Thus,
\begin{equation*}
    f(X)=\sup_{\mu \in ca_c(\mathcal{X})}  \int X d\mu - f^{*}(\mu)\leq \sup_{\mu \in ca_c(\mathcal{X}) } \int X d\mu - \overline f(\mu)=f_D(X)\leq f(X).
\end{equation*}

\begin{remark}
    Recall that the Fenchel-Moreau Theorem (see, e.g., \cite[Part One, Proposition~3.1]{ET1999}) provides necessary and sufficient conditions on $f$ guaranteeing that $f$ admits a dual representation. Indeed, suppose $\mathcal{X}$ carries a locally convex topology $\tau$ such that the dual space $(\mathcal{X},\tau)^\ast$ may be identified with a subset of $ca_c(\mathcal{X})$ via the usual representation of continuous linear functionals as integrals over elements of $ca_c(\mathcal{X})$. Then, according to the Fenchel-Moreau Theorem, $f=f_D$ is equivalent to $f$ being convex, lower semicontinuous with respect to $\tau$, and satisfying that $f$ takes the value $-\infty$ if and only if $f\equiv -\infty$. \hfill $\diamond$
\end{remark}

For $Q \in \mathfrak{P}_{c}(\Omega)$, the \textit{$Q$-dual regularization} of $f$ is given by
\begin{equation*}
    f^{Q}_{D}(X) := \sup_{\mu \in ca_{Q}(\mathcal{X})} \int X d\mu - \overline f^Q(\mu), \quad X \in \mathcal{X},
\end{equation*}
where $ca_Q(\mathcal{X}) :=ca_Q\cap ca_c(\mathcal{X})$ and
\begin{equation*}
    \overline f^Q(\mu):=\sup_{X\in \mathcal{X}}\int Xd\mu - f^Q_E(X), \quad \mu\in ca_{Q}(\mathcal{X}).
\end{equation*}
$f^Q_D$ corresponds to a $Q$-local dual view on $f$. For non-empty $\mathcal{Q}\subseteq \mathfrak{P}_{c}(\Omega)$, we set $f^{\mathcal{Q}}_D := \sup_{Q\in \mathcal{Q}}f^Q_D$. Clearly, $f_D = f^{\mathfrak{P}_{c}(\Omega)}_D$.

\begin{lemma} \label{lem:fD:fE}
    Let $Q\in \mathfrak{P}_c(\Omega)$.
    \begin{enumerate}
        \item[(i)] $f^{Q}_{D}$ is $Q$-consistent.
        
        \item[(ii)] $f^Q_D\leq f^Q_E$.
        
        \item[(iii)] Consider any $Q$-consistent function $f^Q \colon \mathcal{X} \to [-\infty,\infty]$ which satisfies $f^Q \leq f$ and admits a dual representation over $ca_Q(\mathcal{X})$, that is,
        \begin{equation} \label{eq:dual:fQ}
            f^{Q}(X) = \sup_{\mu \in ca_Q(\mathcal{X})} \int X d\mu - (f^Q)^{*}(\mu), \quad X \in \mathcal{X},
        \end{equation}
        where $(f^Q)^{*} \colon ca_Q(\mathcal{X})\to [-\infty,\infty]$. Then $\overline f^Q \leq (f^Q)^{*}$  and $f^Q\leq f^{Q}_{D}$.
        
        \item[(iv)] $\overline f^Q(\mu)=\overline f(\mu)$ for all $\mu\in ca_Q(\mathcal{X})$ and hence,
        \begin{equation*}
            f^{Q}_{D}(X) = \sup_{\mu \in ca_Q(\mathcal{X})} \int X d\mu - \overline{f}(\mu), \quad X \in \mathcal{X}.
        \end{equation*}
        
        \item[(v)] $f^{Q}_{D} \leq f_D \leq f$.    
            
        \item[(vi)] $f^{Q}_{D} = f^{Q}_{E}$ if and only if $f^{Q}_{E}$ admits a dual representation over $ca_Q(\mathcal{X})$.
    \end{enumerate}
\end{lemma}
\begin{proof}
    (i) and (ii) are obvious.
    
    \smallskip \noindent
    (iii) By Lemma~\ref{lem:aux:1}, $f^Q\leq f^Q_E$. Hence, for each $\mu \in ca_Q(\mathcal{X})$, we obtain
    \begin{align*}
        \overline f^Q(\mu) &= \sup_{X \in \mathcal{X}} \int X d\mu - f^Q_E(X) \quad \leq \quad \sup_{X \in \mathcal{X}} \int X d\mu - f^Q(X) \\
        &= \sup_{X \in \mathcal{X}} \int X d\mu - \sup_{\nu \in ca_Q(\mathcal{X})} \bigg(\int X d\nu - (f^Q)^{*}(\nu)\bigg) \quad \leq \quad (f^Q)^{*}(\mu).
    \end{align*}
    Therefore,
    \begin{equation*}
        \int X d\mu - (f^Q)^\ast(\mu)\leq \int X d\mu - \overline f^Q(\mu)\leq f_D^Q(X).
    \end{equation*}
    Taking the supremum over all $\mu \in ca_Q(\mathcal{X})$ on the left side proves the assertion.
    
    \smallskip\noindent
    (iv) Let $g(Y):= \sup_{\mu \in ca_Q(\mathcal{X})} \int Y d\mu - \overline f(\mu)$, $Y\in \mathcal{X}$. By definition, $g \leq f_D\leq f$. Therefore, $\overline f^Q(\mu) \leq \overline f(\mu)$ for all $\mu\in ca_Q(\mathcal{X})$ by (iii). Moreover, for $\mu\in ca_Q(\mathcal{X})$, we have
    \begin{equation*}
        \overline f^Q(\mu)=\sup_{X\in \mathcal{X}}\int X d\mu - f^Q_E(X)\geq \sup_{X\in \mathcal{X}}\int Xd\mu - f(X)=\overline f(\mu)
    \end{equation*}
    since $f_E^Q \leq f$ by Lemma~\ref{lem:aux:1}.
    
    \smallskip\noindent
    (v) follows from (iv).
    
    \smallskip\noindent
    (vi) follows from (ii) and (iii).
\end{proof}

Provided that $\mathcal{X}$ admits a locally convex topology and the corresponding dual space is consistent with $ca_c$, Lemma~\ref{lem:f:fQE}~(iv) implies that $f^Q_D$ is the convex, lower semicontinuous regularization of $f$ (or $f_E^Q$) under $Q$ (see \cite[Part One, Definition~3.2]{ET1999}).

\begin{proposition} \label{prop:fQD}
    Suppose that  $f \colon \mathcal{X} \to [-\infty, \infty]$ admits a dual representation over $ca_c(\mathcal{X})$. Let $\mathcal{Q} \subseteq \mathfrak{P}_{c}(\Omega)$ be non-empty with
    \begin{equation*}
        \{\mu\in ca_c(\mathcal{X})\mid \overline f(\mu)\in \R\} \subseteq \bigcup_{Q \in \mathcal{Q}} ca_{Q}(\mathcal{X}).
    \end{equation*}
    Then $\mathcal{Q}$ is a reduction set for $f$ and $f=f^{\mathcal{Q}}_E=f^{\mathcal{Q}}_D$.
\end{proposition}
\begin{proof}
     By assumption, $f=f_D$. Let $\mathcal{M}:=\{\mu\in ca_c(\mathcal{X})\mid \overline f(\mu)<\infty\}$. Suppose that $\mathcal{M}=\emptyset$. Then $f\equiv -\infty$, and thus $f^Q_D\equiv -\infty$ for all $Q\in \mathfrak{P}_c(\Omega)$. From now on assume that $\mathcal{M}\neq\emptyset$. If there exists $\mu\in ca_c(\mathcal{X})$ such that $\overline f(\mu)=-\infty$, then $f\equiv \infty$. In that case $\overline f\equiv -\infty$, and thus $f^Q_D\equiv \infty$ for all $Q\in \mathfrak{P}_c(\Omega)$ by Lemma~\ref{lem:fD:fE}.  Hence, assume that $\overline f(\mu)>-\infty$ for all $\mu\in ca_c(\mathcal{X})$. 
     Recalling that $f^{Q}_{D} \leq f$ and Lemma~\ref{lem:fD:fE}~(iv), for all $X \in \mathcal{X}$,
    \begin{align*}
        f(X) &= \sup_{\mu \in \mathcal{M}} \int X d\mu - \overline f(\mu)\quad  = \quad \sup_{\mu \in \mathcal{M} \cap \bigcup_{Q \in \mathcal{Q}} ca_{Q}} \int X d\mu - \overline f(\mu) \\
        & =   \sup_{Q \in \mathcal{Q}} \sup_{\mu \in \mathcal{M} \cap ca_{Q}} \int X d\mu - \overline{f}(\mu) \quad \leq \quad   \sup_{Q \in \mathcal{Q}} \sup_{\mu \in ca_{Q}(\mathcal{X})} \int X d\mu - \overline f(\mu) \\ & = \sup_{Q \in \mathcal{Q}} f^{Q}_{D}(X) \quad \leq \quad  f(X).
    \end{align*}
    Thus, $f=f^{\mathcal{Q}}_D$ and $\mathcal{Q}$ is a reduction set for $f$ according to Lemma~\ref{lem:loc:rep}. Finally, recall that always $f\geq f^{\mathcal{Q}}_E\geq f^{\mathcal{Q}}_D$.
\end{proof}   

\begin{theorem} \label{thm:dual:f}    
    Let $f \colon \mathcal{X} \to [-\infty, \infty]$. The following are equivalent: 
    \begin{itemize}
        \item[(i)] $f$ admits a dual representation over $ca_c(\mathcal{X})$.
        
        \item[(ii)] $f=f_D$.
        
        \item[(iii)] There is a non-empty set $\mathcal{Q} \subseteq \mathfrak{P}_c(\Omega)$ such that $f=f^{\mathcal{Q}}_D$.
        
        \item[(iv)] $f$ is $\mathcal{P}$-sensitive and, for a reduction set $\mathcal{Q} \subseteq \mathfrak{P}_c(\Omega)$ of $f$, there exists a $\mathcal{Q}$-localization $(f^{Q})_{Q \in \mathcal{Q}}$ of $f$ such that each $f^Q$ admits a representation \eqref{eq:dual:fQ}.
    \end{itemize}
    In particular, let $\mathcal{Q}\subseteq \mathfrak{P}_c(\Omega)$ be non-empty and suppose that, for all $Q\in \mathcal{Q}$, $j_Q(\mathcal{X})$ admits a topology $\tau^Q$ such that $(j_{Q}(\mathcal{X}), \tau^Q)$ is a locally convex space with dual space $\mathcal{Y}^Q\subseteq ca_Q$. If there is a $\mathcal{Q}$-localization $(f^{Q})_{Q \in \mathcal{Q}}$ of $f$ such that each $f^Q\circ j_Q^{-1}$ (recall Remark~\ref{rem:local}) is convex, lower semicontinuous (with respect to $\tau^Q$), and takes the value $-\infty$ only if $f^Q\equiv -\infty$, then $f$ admits a dual representation over $ca_c(\mathcal{X})$.
\end{theorem}
\begin{proof}
    (i) $\Leftrightarrow$ (ii): Shown above.

    \smallskip\noindent
    (i) $\Rightarrow$ (iii): Suppose that the dual representation of $f$ is given as in \eqref{eq:dualf1}. Consider the set
    \begin{equation*}
        \mathcal{Q} := \bigg\{\frac{|\mu|}{|\mu(\Omega)|} \biggm\vert  \mu \in ca_c(\mathcal{X})\setminus\{0\}\bigg\}\cup\{P\}\subseteq \mathfrak{P}_c(\Omega),
    \end{equation*}
    where $P\in \mathfrak{P}_c(\Omega)$ is arbitrary. For each $\mu\in ca_c(\mathcal{X})$  there is $Q\in \mathcal{Q}$ such that $\mu\ll Q$. Now apply Proposition~\ref{prop:fQD}. 
    
    \smallskip\noindent
    (iii) $\Rightarrow$ (iv): This follows from Theorem~\ref{thm:Psensi:localization}.  
    
    \smallskip \noindent
    (iv) $\Rightarrow$ (ii): By Lemma~\ref{lem:fD:fE}, we have $f^Q\leq f_D^Q$ for all $Q\in \mathcal{Q}$ and therefore,
    \begin{equation*}
        f=\sup_{Q\in \mathcal{Q}}f^Q \leq \sup_{Q\in \mathcal{Q}}f^Q_D \leq f_D \leq f.
    \end{equation*}
    
    Finally, if $(j_Q(\mathcal{X}), \tau^Q)$ is a locally convex space with dual space $\mathcal{Y}^Q\subseteq ca_Q$, then $\mathcal{Y}^Q \subseteq ca_Q(\mathcal{X})$, and the Fenchel-Moreau theorem (see, e.g., \cite[Part One, Proposition~3.1]{ET1999}) implies that, under the stated conditions, $f^{Q}\circ j_Q^{-1}$ admits a dual representation (on $(j_Q(\mathcal{X}), \tau^Q)$), which implies that $f^{Q}$ admits a dual representation \eqref{eq:dual:fQ}. Hence, we have (iv), and therefore (i).
\end{proof}

If one of the conditions of Theorem~\ref{thm:dual:f} is satisfied, then $f=f_D=f_E$ (see Theorem~\ref{thm:Psensi:localization}). 
In Section~\ref{sec:rm}, we show that $f_D^Q<f^Q_E$ is possible even if $f=f_D$, and we discuss sufficient conditions for $f_E^Q=f^Q_D$ tailored to robust monetary risk measures. 

\section{Applications} \label{sec:appl}

\subsection{Robust Optimization}\label{sec:opt:prob}

Let $\mathcal{X}\subseteq L^0_c$ be non-empty. Further, let $f:\mathcal{X}\to [-\infty,\infty]$, and consider the minimization (maximization) problem 
\begin{equation}\label{eq:opt:prob} 
    f(X)\to \min\, (\max) \quad \text{subject to } X \in \mathcal{C},
\end{equation}
where $\mathcal{C}\subseteq\mathcal{X}$ is a non-empty set.

\begin{proposition}\label{prop:opt:1}
    Suppose that $f$ and $\mathcal{C}$ are $\mathcal{P}$-sensitive, and that there exists a joint reduction set $\mathcal{Q}$ for $f$ and $\mathcal{C}$. Let $(f^Q)_{Q\in \mathcal{Q}}$ be a $\mathcal{Q}$-localization of $f$ (see Theorem~\ref{thm:Psensi:localization}). Moreover, suppose that, for each $Q\in \mathcal{Q}$, $X^\ast_Q\in \mathcal{X}$ is a solution to the optimization problem 
    \begin{equation}\label{eq:opt:prob:Q} 
        f^Q(X)\to \min\,  (\max) \quad \mbox{subject to}\; X\in \mathcal{C}.  
    \end{equation}
    Then any $\mathcal{Q}$-aggregator $X^\ast$ of the family $(X^\ast_Q)_{Q\in \mathcal{Q}}$ (if it exists) is a solution to \eqref{eq:opt:prob}.
\end{proposition}

Recalling Remark~\ref{rem:local}, note that solving \eqref{eq:opt:prob:Q} is indeed equivalent to solving 
\begin{equation}\label{eq:opt:prob:Q2} 
f^Q\circ j_Q^{-1}(Y)\to \min\,  (\max) \quad \mbox{subject to}\; Y\in j_Q(\mathcal{C})  
\end{equation}
in the sense that if $X^\ast_Q$ solves \eqref{eq:opt:prob:Q}, then $j_Q(X^\ast_Q)$ solves \eqref{eq:opt:prob:Q2}, and for any solution $Y^\ast$ of \eqref{eq:opt:prob:Q2}, any element in $j_Q^{-1}(Y)\cap \mathcal{C}$ is a solution to \eqref{eq:opt:prob:Q}.  \eqref{eq:opt:prob:Q2} is an optimization problem on a subset of $L^0_Q$ under the dominating probability measure $Q$. Here, a classical machinery for solving such problems is available, in particular if each $f^Q$ is convex and admits a dual representation \eqref{eq:dual:fQ}.

\begin{proof}[Proof of Proposition~\ref{prop:opt:1}]
    Since $\mathcal{C}$ is $\mathcal{Q}$-stable by Lemma~\ref{lem:Qstable}, we conclude that $X^\ast\in \mathcal{C}$. Suppose that $Y\in \mathcal{C}$ satisfies $f(Y) < f(X^\ast)$ ($f(Y) > f(X^\ast)$). Then, as $f=\sup_{Q\in \mathcal{Q}}f^Q$, there is $Q\in \mathcal{Q}$ such that
    \begin{equation*}
        f^Q(Y)< f^Q(X^\ast)=f^Q(X^\ast_Q) \quad \big(f^Q(Y)> f^Q(X^\ast)=f^Q(X^\ast_Q)\big)
    \end{equation*}
    where the last equality holds due to the $Q$-consistency of $f^Q$. This contradicts the assumed optimality of $X^\ast_Q$.
\end{proof}

\begin{corollary}
    Suppose that $f$ is upper $\mathcal{P}$-sensitive and $\mathcal{C}$ is $\mathcal{P}$-sensitive, and that there exists a joint reduction set $\mathcal{Q}$ for $\mathcal{C}$ and all upper level sets of $f$. Let $(f^Q)_{Q\in \mathcal{Q}}$ be a $\mathcal{Q}$-consistent family of functions such that \eqref{eq:inf:loc} holds. Moreover, suppose that for each $Q\in \mathcal{Q}$, $X^\ast_Q\in \mathcal{X}$ is a solution to \eqref{eq:opt:prob:Q}.
    Then any $\mathcal{Q}$-aggregator $X^\ast$ of the family $(X^\ast_Q)_{Q\in \mathcal{Q}}$ (if it exists) is a solution to \eqref{eq:opt:prob}.
\end{corollary}
\begin{proof}
    Apply Proposition~\ref{prop:opt:1} to $-f$.
\end{proof}

In view of Proposition~\ref{prop:opt:1}, in Lemma~\ref{lem:aggre:always} we give a condition which guarantees that there always exists a $\mathcal{Q}$-aggregator $X^\ast$ of the family $(X^\ast_Q)_{Q\in \mathcal{Q}}$ if $(X^\ast_Q)_{Q\in \mathcal{Q}}$ is sufficiently bounded. To this end, we will need to distinguish between so-called supported and unsupported probability measures: 

\begin{definition} \label{defi:supp}
    A probability measure $Q\in \mathfrak{P}_c(\Omega)$ is called \textit{supported} if there is an event $S(Q) \in \mathcal{F}$ such that
    \begin{enumerate}[(i)]
        \item $Q(S(Q)^{c}) = 0$,
        
        \item whenever $N \in \mathcal{F}$ satisfies $Q(N) = 0$, then $N \cap S(Q)$ is $\mathcal{P}$-polar.
    \end{enumerate}
    The set $S(Q)$ is called the \textit{(order) support} of $Q$.
\end{definition}

Supported measures play a key role in handling robustness. Standard robust models, such as volatility uncertainty models, are based on supported probability measures, see \cite{LMS2022} for a detailed review. Note that if two sets $S, S' \in \mathcal{F}$ satisfy conditions (i) and (ii) in Definition \ref{defi:supp}, then $\chi_{S} = \chi_{S'}$ $\mathcal{P}$-q.s. (${\bf 1}_{S}={\bf 1}_{S'}$), i.e., the symmetric difference $S \bigtriangleup S'$ is $\mathcal{P}$-polar. The order support $S(Q)$ is thus usually not unique as an event, but only unique up to $\mathcal{P}$-polar events. In the following, $S(Q)$ therefore denotes an arbitrary version of the order support.

We recall that $\mathcal{X}$ is said to be \textit{Dedekind complete} if every bounded set $\mathcal{D}\subset \mathcal{X}$ admits a least upper bound in $\mathcal{X}$. For a detailed discussion of Dedekind completeness of robust model spaces, we again refer to \cite{LMS2022}. Versions of Lemma~\ref{lem:aggre:always} and Corollary~\ref{cor:aggre:always:1} are found in the literature (see, e.g., \cite{L1978}). We state their short proofs for the sake of completeness.

\begin{lemma}\label{lem:aggre:always}
    Suppose that $\mathcal{X}$ is Dedekind complete and let $\mathcal{Q}\subset \mathfrak{P}_c(\Omega)$  be non-empty. Moreover, suppose that each $Q\in\mathcal{Q}$ is supported. Consider any family $(X^Q)_{Q\in \mathcal{Q}}\subseteq \mathcal{X}$ which is pairwise $\mathcal{Q}$-coherent, that is, for all $Q,Q'\in \mathcal{Q}$, we have
    \begin{equation*}
        X^Q\mathbf{1}_{S(Q)\cap S(Q')}=X^{Q'}\mathbf{1}_{S(Q)\cap S(Q')},
    \end{equation*}
    and bounded in the sense that there exists some $Y\in \mathcal{X}$ such that $|X^Q|\mathbf{1}_{S(Q)}\preccurlyeq Y$. Then  $(X^Q)_{Q\in \mathcal{Q}}$ is $\mathcal{Q}$-coherent (and thus admits a $\mathcal{Q}$-aggregator).
\end{lemma}

\begin{proof}
    As the family $Y^Q:=X^Q\mathbf{1}_{S(Q)}-Y\mathbf{1}_{S(Q)^c}$, $Q\in\mathcal{Q}$, is bounded by $Y$, there exists a least upper bound $X$ of $(Y^Q)_{Q\in\mathcal{Q}}$. We will show that $j_Q(X)=j_Q(X^Q)$ for all $Q\in\mathcal{Q}$. To this end, let $Q'\in\mathcal{Q}$. By $X$ being an upper bound, we know that $j_{Q'}(X^{Q'})=j_{Q'}(Y^{Q'})\leq_{Q'} j_{Q'}(X)$. Suppose that $Q'(X^{Q'}< X)>0$. Let $Z := X^{Q'}\mathbf{1}_{S(Q')} + X\mathbf{1}_{S(Q')^c}$. Then $Z\preccurlyeq X$ and $Z\neq X$. Clearly, by definition of $X$, we have that
    \begin{equation*}
        \mathbf{1}_{S(Q')^{c}} Y^{Q} \preccurlyeq \mathbf{1}_{S(Q')^{c}} X = \mathbf{1}_{S(Q')^{c}} Z
    \end{equation*}
    for all $Q \in \mathcal{Q}$. To see that $Y^{Q} \mathbf{1}_{S(Q')} \preccurlyeq Z \mathbf{1}_{S(Q')}$ holds for all $Q \in \mathcal{Q}$, we first note that $-Y \preccurlyeq Z$ and thus
    \begin{equation*}
        Y^{Q} \mathbf{1}_{S(Q') \cap S(Q)^{c}} = -Y \mathbf{1}_{S(Q') \cap S(Q)^{c}} \preccurlyeq Z \mathbf{1}_{S(Q') \cap S(Q)^{c}}.
    \end{equation*}
    On $S(Q')\cap S(Q)$, by pairwise $\mathcal{Q}$-coherence of $(X^{Q})_{Q \in \mathcal{Q}}$, we have
    \begin{equation*}
        Y^Q \mathbf{1}_{S(Q') \cap S(Q)} = X^Q \mathbf{1}_{S(Q') \cap S(Q)} = X^{Q'}\mathbf{1}_{S(Q') \cap S(Q)} = Z \mathbf{1}_{S(Q') \cap S(Q)}.
    \end{equation*}
    Overall, it then follows that for all $Q \in \mathcal{Q}$,
    \begin{align*}
        Y^{Q} &= \mathbf{1}_{S(Q') \cap S(Q)} Y^{Q} + \mathbf{1}_{S(Q') \cap S(Q)^{c}} Y^{Q} + \mathbf{1}_{S(Q')^{c}} Y^{Q} \\
        &\preccurlyeq \mathbf{1}_{S(Q') \cap S(Q)} Z + \mathbf{1}_{S(Q') \cap S(Q)^{c}} Z + \mathbf{1}_{S(Q')^{c}} Z = Z
    \end{align*}
    Hence, $Z$ is an upper bound of $(Y^{Q})_{Q \in \mathcal{Q}}$ with $Z \preccurlyeq X$ and $Z\neq X$, contradicting the fact that $X$ is the least upper bound. Therefore, we must have $Q'(X^{Q'}< X)=0$ which shows that indeed $j_{Q'}(X)=j_{Q'}(X^{Q'})$.
\end{proof}

Note that, when $\mathcal{Q}$ is uncountable (which is the interesting case), the assertion of Lemma~\ref{lem:aggre:always}, that pairwise $\mathcal{Q}$-coherence implies $\mathcal{Q}$-coherence, is not as obvious as it may seem on a first glance. 

\begin{definition}
    Supported probability measures $Q, Q' \in \mathfrak{P}_c(\Omega)$ are called \textit{disjoint} if $S(Q)\cap S(Q')$ is a $\mathcal{P}$-polar set (for any choice of $S(Q)$ and $S(Q')$) or, equivalently, $\mathbf{1}_{S(Q)} \wedge \mathbf{1}_{S(Q')} = 0$.
\end{definition}

In fact, many robust models assume disjoint supported reference measures,  see \cite{LMS2022}.

\begin{corollary}\label{cor:aggre:always:1}
    Suppose that $\mathcal{X}$ is Dedekind complete and let $\mathcal{Q}\subset \mathfrak{P}_c(\Omega)$  be non-empty. Moreover, suppose that each $Q\in\mathcal{Q}$ is supported and the elements of $\mathcal{Q}$ are mutually disjoint. Consider any family $(X^Q)_{Q\in \mathcal{Q}}\subseteq \mathcal{X}$  which is bounded in the sense that there is $Y\in \mathcal{X}$ with $|X^Q|\mathbf{1}_{S(Q)}\preccurlyeq Y$. Then  $(X^Q)_{Q\in \mathcal{Q}}$ is $\mathcal{Q}$-coherent (and thus admits a $\mathcal{Q}$-aggregator).
\end{corollary}
\begin{proof}
    Since $\mathbf{1}_{S(Q)\cap S(Q')}=0$ for all $Q,Q'\in \mathcal{Q}$, every family $(X^Q)_{Q\in \mathcal{Q}}\subseteq \mathcal{X}$ is pairwise $\mathcal{Q}$-coherent. By Lemma~\ref{lem:aggre:always}, the result follows.
\end{proof}

\begin{example}[Bliss point consumption]
    Suppose that $\mathcal{X}$ is Dedekind complete and each $P \in \mathcal{P}$ is supported. Let $A,B \in \mathcal{X}$ such that $A\preccurlyeq B$ and set $\mathcal{C} := \{X \in \mathcal{X} \mid A \preccurlyeq X \preccurlyeq B\}$. We interpret each $X \in \mathcal{C}$ as an admissible consumption plan and $A,B \in \mathcal{X}$ as  global lower and upper constraints. Note that $\mathcal{C}$ is $\mathcal{P}$-sensitive with reduction set $\mathcal{P}$. Indeed, let $X\in \mathcal{X}$ satisfy $j_P(X)\in j_P(\mathcal{C})$, that is $A\leq_P X\leq_P B$, for all $P\in \mathcal{P}$, then, by definition of the $\mathcal{P}$-q.s.\ order, $A\preccurlyeq X \preccurlyeq B$, so $X\in \mathcal{C}$.
    
    Locally, that is, under each reference measure $P \in \mathcal{P}$,  $Y^{P} \in \mathcal{X}$ denotes a target consumption level. We suppose that the family $(Y^{P})_{P \in \mathcal{P}}$ is pairwise $\mathcal{P}$-coherent. The task is to find a consumption plan $X^\ast\in \mathcal{C}$ that minimizes 
    \begin{equation}\label{eq:opt:ex}
        f(X) := \sup_{P \in \mathcal{P}} E_{P}\big[(Y^{P} - X)^{2}\big]\quad \mbox{subject to } \; X \in \mathcal{C}.
    \end{equation}
    For each $P \in \mathcal{P}$, define $f^{P} \colon \mathcal{X} \to [-\infty, \infty]$ by
    \begin{equation*}
        f^{P}(X) := E_{P}\big[(Y^{P} - X)^{2}\big].
    \end{equation*}
    Clearly, $f$ is $\mathcal{P}$-sensitive with reduction set $\mathcal{P}$ (see Lemma~\ref{lem:loc:rep}) and $(f^{P})_{P \in \mathcal{P}}$ is a $\mathcal{P}$-localization of $f$. Further, we obtain the local optimizers
    \begin{equation*}
        X_{P}^{*} := A \mathbf{1}_{\{Y^{P} \preccurlyeq A\}} + Y^{P} \mathbf{1}_{\{A \preccurlyeq Y^{P} \preccurlyeq B, A\neq Y^P\}} + B \mathbf{1}_{\{B \preccurlyeq Y^{P}, Y^P\neq B\}}  \in \mathcal{C}.
    \end{equation*}
    Fix $P, P' \in \mathcal{P}$. Then, since $(Y^{P})_{P \in \mathcal{P}}$ is pairwise $\mathcal{P}$-coherent, we have
    \begin{align*}
        X_{P}^{*} \mathbf{1}_{S(P) \cap S(P')} &= \big(A \mathbf{1}_{\{Y^{P} \preccurlyeq A\}} + Y^{P} \mathbf{1}_{\{A \preccurlyeq Y^{P} \preccurlyeq B, A\neq Y^P\}} + B \mathbf{1}_{\{B \preccurlyeq Y^{P}, Y^P\neq B\}}\big) \mathbf{1}_{S(P) \cap S(P')} \\
        &= \big(A \mathbf{1}_{\{Y^{P'} \preccurlyeq A\}} + Y^{P'} \mathbf{1}_{\{A \preccurlyeq Y^{P'} \preccurlyeq B, A\neq Y^{P'}\}} + B \mathbf{1}_{\{B \preccurlyeq Y^{P'}, Y^{P'}\neq B\}}\big) \mathbf{1}_{S(P) \cap S(P')} \\
        &= X_{P'}^{*} \mathbf{1}_{S(P) \cap S(P')}.
    \end{align*}
    Hence, $(X_{P}^{*})_{P \in \mathcal{P}}$ is pairwise $\mathcal{P}$-coherent and $\lvert X_{P}^{*} \rvert \mathbf{1}_{S(P)} \preccurlyeq \lvert X_{P}^{*} \rvert \preccurlyeq \lvert A \rvert \vee \lvert B \rvert$ for all $P \in \mathcal{P}$. By Lemma~\ref{lem:aggre:always}, there exists a $\mathcal{P}$-aggregator $X^{*}$ of the family $(X_{P}^{*})_{P \in \mathcal{P}}$. 
     Consequently, $X^{*}$ solves \eqref{eq:opt:ex} according to Proposition~\ref{prop:opt:1}. \hfill $\diamond$
\end{example}

\subsection{Monetary Risk Measures and Localization Bubbles} \label{sec:rm}

In this section, $\mathcal{X}=L^\infty_c$ and thus $ca_c(\mathcal{X})=ca_c$.

\begin{definition}
    A mapping $\rho \colon L^\infty_c \to \R$ is called a \textit{monetary risk measure} if it is monotone and cash-additive (see Section~\ref{sec:properties}). $\mathcal{A}_{\rho} := \{X \in L^\infty_c \mid \rho(X) \leq 0\}$ denotes the \textit{acceptance set} of $\rho$. If $\rho$ is, in addition, convex, $\rho$ is called a \textit{convex risk measure}. A convex risk measure which is also positively homogeneous is called a \textit{coherent risk measure}.
\end{definition}

$\rho(X)$ is a capital requirement which may also be seen as the indifference price of the loss $X$. In fact, charging $\rho(X)$ reduces the loss to $X-\rho(X)$, which is acceptable since $\rho(X-\rho(X)) = 0$ by cash-additivity. Let $Q \in \mathfrak{P}_c(\Omega)$. We may interpret $\rho^Q_E$ as the minimal price from a $Q$-perspective charged by the market with aggregate risk measure $\rho$ for taking a loss $X$. Indeed, from a $Q$-perspective, all losses $Y$ such that $j_Q(Y)=j_Q(X)$ are the same. Thus, when we hand the loss $X$ to the market, we buy a contract $Y$ such that $j_Q(Y)=j_Q(X)$---which therefore secures $X$ from the $Q$ perspective---at minimal price. In other words, the price is $\inf\{\rho(Y) \mid j_Q(Y)=j_Q(X)\}$ (or arbitrarily close when this infimum is not attained), which is exactly $\rho^Q_E(X)$ according to Lemma~\ref{lem:aux:1}. By Lemma~\ref{lem:f:fQE} and Lemma~\ref{lem:rhoE:rm} below, $\rho^Q_E$ is a monetary risk measure provided that $\rho^Q_E(0)\in \R$. Otherwise, we are in a degenerate situation. 

\begin{definition}
    Let $\rho \colon L^{\infty}_{c} \to \R$ be a monetary risk measure and $Q \in \mathfrak{P}_{c}(\Omega)$. We call $\rho^{Q}_{E}$ \textit{relevant} if $\rho^Q_E(0)\in \R$.
\end{definition}

\begin{lemma}\label{lem:rhoE:rm}
    Let $\rho \colon L^{\infty}_{c} \to \R$ be a monetary risk measure and $Q \in \mathfrak{P}_{c}(\Omega)$. Then the following are equivalent:
    \begin{enumerate}
        \item[(i)] $\rho^{Q}_{E}$ is a monetary risk measure.
        
        \item[(ii)] $\rho^{Q}_{E}$ is relevant.
        
        \item[(iii)] $\sup \{m \in \R \mid m \in j_{Q}(\mathcal{A}_\rho)\} < \infty$.
    \end{enumerate}
\end{lemma}
\begin{proof}
    (i) obviously implies (ii). In order to prove that (ii) implies (i), recall that according to Lemma~\ref{lem:f:fQE}, $\rho^{Q}_{E}$ is monotone and cash-additive, and thus a monetary risk measure on $L^\infty_c$, provided that $\rho_E^Q$ is real-valued. If $\rho^{Q}_{E}$ is relevant, then by cash-additivity, $\rho_E^Q(m)=\rho^Q_E(0)+m\in \R$ for all $m\in \R$. Moreover, since
    \begin{equation*}
        \rho_E^Q(-\|X\|_{c,\infty})\leq \rho_E^Q(X)\leq \rho_E^Q(\|X\|_{c,\infty})
    \end{equation*}
    for all $X\in L^\infty_c$ by monotonicity, it follows that $\rho_E^Q$ is indeed real-valued.

    \smallskip\noindent
    (ii) $\Leftrightarrow$ (iii) follows from  $E_r =\{X\in L^\infty_c\mid \rho(X)\leq r \}= \mathcal{A}_\rho+r$ and \begin{align*}
        \rho^Q_E(0) &= \inf \{m \in \R \mid 0 \in j_{Q}(\mathcal{A}_\rho + m)\} \\
        &= \inf \{m \in \R \mid -m \in j_{Q}(\mathcal{A}_\rho)\} \\
        &= -\sup \{m \in \R \mid m \in j_{Q}(\mathcal{A}_\rho)\},
    \end{align*}
    and $\rho_E^Q(0)\leq \rho(0)$ (see Lemma~\ref{lem:aux:1}).
\end{proof}

Note that if $\sup \{m \in \R \mid m \in j_{Q}(\mathcal{A}_\rho)\} = \infty$, that is, $\rho^Q_E(0) = -\infty$, then, by cash-additivity and monotonicity, it follows that $j_{Q}(\mathcal{A}_\rho) = L^{\infty}_{Q}$ and $\rho^{Q}_{E}\equiv -\infty$.

\begin{remark}\label{rem:rmonLQ}
    Lemma~\ref{lem:f:fQE} not only shows that $\rho^Q_E$ is a monetary risk measure on $L^\infty_c$ provided it is relevant, but also, thanks to the monotonicity established in Lemma~\ref{lem:f:fQE}~(i), that $\rho^Q_E\circ j_Q^{-1}$ (recall Remark~\ref{rem:local}) is indeed a monetary risk measure on $L^\infty_Q$ which is convex or even coherent whenever $\rho$ is. \hfill $\diamond$
\end{remark}

\begin{lemma} \label{lem:rhoD}
    Let $\rho \colon L^{\infty}_{c} \to \R$ be a monetary risk measure. Then $\{\mu \in ca_c\mid \overline{\rho}(\mu)<\infty\} \subseteq \mathfrak{P}_{c}(\Omega)$ and therefore,
    \begin{equation*}
        \rho_{D}(X) = \sup_{Q \in \mathfrak{P}_{c}(\Omega)} E_{Q}[X] - \overline{\rho}(Q), \quad X \in L^{\infty}_{c},
    \end{equation*}
    where
    \begin{equation*}
        \overline{\rho}(Q)=\sup_{X\in \mathcal{A}_\rho} E_Q[X], \quad Q\in \mathfrak{P}_c(\Omega).
    \end{equation*}
    Consequently,
    \begin{equation*}
        \rho^{Q}_{D}(Y) = \sup_{R \ll Q} E_{R}[Y] - \overline{\rho}(R), \quad Y\in L^\infty_Q.
    \end{equation*}
\end{lemma}
\begin{proof}
    $\{\mu \in ca_c\mid \overline{\rho}(\mu)<\infty\} \subseteq \mathfrak{P}_{c}(\Omega)$ follows from a straightforward adaptation of (parts of) the proof of \cite[Theorem 4.16]{FS2016}.
\end{proof}

\begin{lemma} \label{lem:rhoQE:rep}
    Let $\rho \colon L^{\infty}_{c} \to \R$ be a monetary risk measure and let $Q \in \mathfrak{P}_c(\Omega)$ be supported (recall Definition~\ref{defi:supp}). Then for any $X \in L^{\infty}_c$,
    \begin{equation*}
        \rho^{Q}_{E}(X) = \lim_{m \to \infty} \rho\big(X \mathbf{1}_{S(Q)} - m \mathbf{1}_{S(Q)^{c}}\big).
    \end{equation*}
    If $\rho=\rho_D$, then $\rho^{Q}_{E}$ is relevant only if $\inf_{\overline{\rho}(R)<\infty}R(S(Q)^c)=0$. Moreover, if $\rho$ is coherent and $\rho=\rho_D$, then $\rho_E^Q$ is relevant if and only if $\inf_{\overline{\rho}(R)<\infty}R(S(Q)^c)=0$.
\end{lemma}
\begin{proof}
     Let  $X_{m} := X \mathbf{1}_{S(Q)} - m \mathbf{1}_{S(Q)^{c}}$, $m \in \N$. Then $j_Q(X)=j_Q(X_{m})$. For any $Y\in  L^{\infty}_{c}$ such that $j_Q(Y)=j_Q(X)$, choosing $m \geq \lVert Y \rVert_{c,\infty}$, we obtain $X_{m} \preccurlyeq Y$. Therefore, by monotonicity, it follows that
    \begin{equation*}
        \rho_E^Q(X)=\inf\{\rho(Y)\mid j_Q(Y)=j_Q(X)\}= \inf_{m\geq 0 } \rho(X_m)=\lim_{m\to \infty} \rho(X_{m}).
    \end{equation*}
    In particular, assuming that $\rho=\rho_D$ and recalling Lemma~\ref{lem:rhoD}, we have
    \begin{align*}
        \rho^{Q}_{E}(0) &= \lim_{m \to \infty} \rho\big(-m \mathbf{1}_{S(Q)^{c}}\big) \\
        &= \lim_{m \to \infty} \sup_{\overline{\rho}(R) < \infty} - m R\big(S(Q)^c\big) - \overline{\rho}(R) \\
        &\leq \rho(0) + \lim_{m \to \infty} \sup_{\overline{\rho}(R) < \infty}- m R\big(S(Q)^{c}\big) \\
        &= \rho(0) + \lim_{m \to \infty} - \Big(m \inf_{\overline{\rho}(R) < \infty} R\big(S(Q)^{c}\big)\Big),
    \end{align*}
    where we used the estimate $\overline{\rho}(R) \geq E_{R}[0] - \rho(0) = -\rho(0)$. Therefore, $\rho^{Q}_{E}$ is relevant only if $\inf_{\overline{\rho}(R) < \infty}R(S(Q)^{c}) = 0$. In case $\rho$ is coherent, the inequality above is an equality which yields the claimed sufficiency.
\end{proof}

\begin{lemma} \label{lem:rel:meas}
    Let $\rho \colon L^{\infty}_{c} \to \R$ be a monetary risk measure, and let $$\mathcal{Q}_{rel}^\rho:=\{Q\in \mathfrak{P}_c(\Omega)\mid \rho^{Q}_{E}\, \mbox{is relevant}\}.$$ If $\mathcal{A}_\rho$ is $\mathcal{P}$-sensitive, then $\mathcal{Q}_{rel}^\rho$ is a reduction set for $\mathcal{A}_\rho$, and $\rho=\rho^{\mathcal{Q}_{rel}^\rho}_E$. If, in addition, $\rho=\rho_D$, then $\rho=\rho^{\mathcal{Q}_{rel}^\rho}_E= \rho^{\mathcal{Q}_{rel}^\rho}_D$.
\end{lemma}
\begin{proof}
    Let $\mathcal{A}_\rho$ be $\mathcal{P}$-sensitive. Suppose that for all $Q \in \mathfrak{P}_{c}(\Omega)$, the risk measures $\rho^{Q}_{E}$ are not relevant, and thus $j_{Q}(\mathcal{A}_\rho)=L^\infty_Q$ for all $Q\in \mathfrak{P}_c(\Omega)$. Let $m \in \R$. Then, trivially, $m \in j_{Q}(\mathcal{A}_\rho)$ for all $Q \in \mathfrak{P}_c(\Omega)$. Consequently, $\mathcal{P}$-sensitivity implies that $m \in \mathcal{A}_\rho$. Thus, $-\rho(0)=\sup\{m \in \R \mid m \in \mathcal{A}_\rho\} = \infty$, which is a contradiction since $\rho(0) \in \R$. Hence, we must have that $\mathcal{Q}_{rel}^\rho\neq \emptyset$. Moreover, by definition of $\mathcal{P}$-sensitivity, we have
    \begin{equation*}
        X\in \mathcal{A}_\rho \quad \Leftrightarrow \quad \forall Q\in \mathfrak{P}_c(\Omega) \colon \, j_Q(X) \in j_Q(\mathcal{A}_\rho).
    \end{equation*}
    Those $Q$ such that $j_{Q}(\mathcal{A}_\rho) = L^\infty_Q$ do not pose any constraint on the right-hand side. Therefore, if $j_Q(X) \in j_Q(\mathcal{A}_\rho)$ for all $Q \in \mathcal{Q}_{rel}^\rho$, then indeed $j_Q(X) \in j_Q(\mathcal{A}_\rho)$ for all $Q \in \mathfrak{P}_c(\Omega)$, and thus $X \in \mathcal{A}_\rho$. Hence, $\mathcal{Q}_{rel}^\rho$ is a reduction set for $\mathcal{A}_\rho$. Theorem~\ref{thm:Psensi:localization} and Remark~\ref{rem:rhoP} below imply that $\rho = \rho^{\mathcal{Q}_{rel}^\rho}_E$.

    Suppose that $\rho=\rho_D$. Let $Q\in \mathfrak{P}_c(\Omega)$ such that $\overline \rho(Q)\in \R$. Then  $\rho^Q_E(0) \geq \rho^Q_D(0)\geq -\overline \rho(Q)$ (see Lemma~\ref{lem:rhoD}). Since also $\rho^Q_E(0)\leq \rho(0)$, we obtain $\rho^Q_E(0)\in \R$, and therefore $Q\in \mathcal{Q}_{rel}^\rho$. Thus, $\rho = \rho^{\mathcal{Q}_{rel}^\rho}_E = \rho^{\mathcal{Q}_{rel}^\rho}_D$ follows from Proposition~\ref{prop:fQD}.
\end{proof}

\begin{remark}\label{rem:rhoP}
    The level sets $E_r$, $r \in \R$, of any monetary risk measure $\rho$ satisfy $E_r = \mathcal{A}_\rho+r$. Therefore, $\rho$ is $\mathcal{P}$-sensitive if and only if $\mathcal{A}_\rho$ is $\mathcal{P}$-sensitive, and any reduction set for $\mathcal{A}_\rho$ is a reduction set for $\rho$, and vice versa. \hfill $\diamond$
\end{remark}

Regarding an interpretation of $\rho_D^Q$, note that measures $R \in \mathfrak{P}_{c}(\Omega)$ such that $\overline{\rho}(R) < \infty$ correspond to acceptability constraints, since $X\in \mathcal{A}_\rho$ if and only if $E_R[X]\leq \overline{\rho}(R)$ for all such $R$. They may also be interpreted as penalized pricing functionals. Measures $R$ such that also $R \ll Q$ correspond to constraints which are understood from the $Q$-perspective. Hence, the convex risk measure $\rho_D^Q(X)$ is the capital requirement or indifference price of the loss $X$ based on acceptability constraints from the $Q$-perspective.

Assume that the monetary risk measure $\rho$ satisfies $\rho=\rho_D$, i.e., it displays a reasonable degree of regularity from a dual perspective, a property that is commonly assumed. Models $Q \not\in \mathcal{Q}_{rel}^\rho$ are irrelevant, whereas $Q \in \mathcal{Q}_{rel}^\rho$ are the models that in aggregate yield $\rho$ both from a primal and a dual perspective, since $\rho = \rho^{\mathcal{Q}_{rel}^\rho}_E = \rho^{\mathcal{Q}_{rel}^\rho}_D$ according to Lemma~\ref{lem:rel:meas}. In our discussion we will focus on $Q\in  \mathcal{Q}_{rel}^\rho$, even if the results derived below do not require this condition. Recall that always $\rho_D^Q\leq \rho_E^Q$. Examples~\ref{ex:rhoD:rhoE} and~\ref{ex:rhoD:rhoE:2} show that $\rho_D^Q(X)< \rho_E^Q(X)$ is possible. In that case, $\rho_E^Q(X)-\rho_D^Q(X)$ may be seen as a market inconsistency which we call a \textit{localization bubble}, where the actual price under $Q$ exceeds what the risk assessment based on acceptability constraints under $Q$ would imply. In Example~\ref{ex:rhoD:rhoE}, the size of this bubble is infinite, whereas in Example~\ref{ex:rhoD:rhoE:2}, it is finite. In the following, we will show how such a bubble is related to purely finitely additive set functions, so-called \textit{charges}, appearing in the dual representation of $\rho^Q_E$. Note that there is a line of literature that identifies bubbles as a consequence of pricing under charges (see, e.g., \cite{B1972}, \cite{G1989}, \cite{GL1992}, \cite{JPS2010}, and references therein).

If $\rho$ is convex and $Q\in \mathcal{Q}_{rel}^\rho$, so that $\rho^{Q}_{E}$ is relevant, then $\rho^{Q}_{E}\circ j_Q^{-1}$ is a convex risk measure on $L^\infty_Q$ (recall Remark~\ref{rem:rmonLQ}). Hence, it is well-known that we obtain a dual representation of $\rho_E^Q$ over the finitely additive measures
\begin{equation*}
    ba_Q^1 := \{\mu \in ba_Q \mid \mu(\Omega) = 1  \text{ and } \mu(A) \geq 0 \text{ for all } A \in \mathcal{F}\},
\end{equation*}
where $ba$ is the real vector space of all finitely additive finite variation set functions $\mu \colon \mathcal{F} \rightarrow \mathbb{R}$, and
\begin{equation*}
    ba_{Q} := \{\mu \in ba \mid |\mu| \ll Q\}.
\end{equation*}
The total variation of finitely additive measures is defined in the same way as for measures in \eqref{eq:tv}.
In fact (see, e.g., \cite[Theorem 4.16]{FS2016}), 
\begin{equation}\label{eq:rho:Q:E:dual}
    \rho^{Q}_{E}(X) = \max_{\mu \in ba_Q^1} \int X d\mu  - \rho^\ast_Q(\mu), \ \ X \in L^{\infty}_{c},
\end{equation}
where the dual function $\rho^\ast_Q$ is given by
\begin{equation*}
   \rho^\ast_Q(\mu) := \sup\bigg\{\int Y d\mu\biggm\vert Y \in L^{\infty}_c \colon \rho_E^Q (Y)\leq 0\bigg\}, \quad \mu \in ba_{Q}^{1}.
\end{equation*}
Note that in \eqref{eq:rho:Q:E:dual} we used the $Q$-consistency of $\rho^Q_E$ to rewrite the dual representation of $\rho^Q_E\circ j_Q^{-1}$ on $L^\infty_Q$ as a representation of $\rho^Q_E$ on $L^\infty_c$. Consider any $R\in \mathfrak{P}_Q(\Omega)$. Then,
\begin{equation*}
    \overline{\rho}(R) = \sup_{Y \in \mathcal{A}_{\rho}} E_{R}[Y]= \sup_{Y \in L^{\infty}_{c} \colon \rho^{Q}_{E}(Y) \leq 0} E_{R}[Y] = \rho^{\ast}_{Q}(R),
\end{equation*}
where for the second equality, we used that $Y \in \mathcal{A}_{\rho}$ implies $\rho^{Q}_{E}(Y) \leq 0$, and, conversely,
\begin{equation*}
    \rho^{Q}_{E}(Y)=\inf\{\rho(Z)\mid j_Q(Z)=j_Q(Y)\} \leq 0
\end{equation*}
implies that there is a non-increasing sequence $(m_{n})_{n \in \N} \subseteq \R_{+}$ such that $\lim_{n \to \infty} m_{n} \leq 0$, and a sequence $(Y_n)_{n \in \N} \subseteq L^\infty_c$ with $j_Q(Y_n) = j_{Q}(Y)$ and $Y_n - m_{n} \in \mathcal{A}_{\rho}$ for all $n \in \N$. Note that, as $R\ll Q$, $E_R[Y_n]=E_R[Y]$. 
Hence, if $\rho^{Q}_{E}(Y) > \rho^{Q}_{D}(Y)$, then there must exist $\nu \in ba_{Q}^{1} \setminus \mathfrak{P}_Q(\Omega)$ such that
\begin{equation} \label{eq:ba}
    \rho^{Q}_{E}(Y) = \int Y d\nu - \rho^{*}_{Q}(\nu) > \rho^Q_{D}(Y).
\end{equation}
This establishes the aforementioned relation of localization bubbles to charges.  

Before stating Examples~\ref{ex:rhoD:rhoE} and \ref{ex:rhoD:rhoE:2}, we give a number of sufficient conditions ensuring that localization bubbles do not appear. An obvious sufficient condition is that $Q$ be finitely atomic, that is, there are pairwise disjoint events $A_1,\ldots, A_n\in \mathcal{F}$ such that $Q(A_i)>0$ for all $i=1,\ldots, n$, $\bigcup_{i=1}^nA_i=\Omega$, and, for any $B\in \mathcal{F}$ and any $i\in \{1,\ldots, n\}$, $B\subseteq A_i$ implies $Q(B)=0$ or $Q(B)=Q(A_i)$. In that case $ba^1_Q=\mathfrak{P}_Q(\Omega)$. Hence, \eqref{eq:ba} cannot hold.

\begin{lemma}
    Let $\rho:L^\infty_c\to \R$ be a convex risk measure and suppose that $Q\in \mathfrak{P}_c(\Omega)$ is finitely atomic. Then $\rho^Q_E=\rho^Q_D$.
\end{lemma}

The next sufficient condition for $\rho^Q_E=\rho^Q_D$ is a version of the Lebesgue property discussed, e.g., in \cite[Corollary 4.35]{FS2016}.

\begin{proposition}
    Let $\rho \colon L^{\infty}_{c} \to \R$ be a coherent risk measure. Suppose that $\rho(\mathbf{1}_{A_n}) \downarrow 0$ whenever the sequence of events $(A_n)_{n \in \N}\subseteq \mathcal{F}$ satisfies $A_n\downarrow \emptyset$. Then $\rho^Q_E=\rho^Q_D$ for all $Q\in \mathfrak{P}_c(\Omega)$.
\end{proposition}
\begin{proof}
    If $Q\not \in \mathcal{Q}_{rel}^\rho$, then $\rho^Q_E\equiv-\infty$ and thus also $\rho^Q_D\equiv -\infty$ (recall $\rho^Q_D\leq \rho^Q_E$). Let $Q\in \mathcal{Q}_{rel}^\rho$. Let $\nu\in ba^1_Q\setminus \mathfrak{P}_c(\Omega)$. Then there exists $(A_n)_{n\in \N}\subseteq \mathcal{F}$ such that $A_n\downarrow \emptyset$ and $\varepsilon>0$ such that $\nu(A_n)>\varepsilon$ for all $n\in \N$ (see \cite[Lemma 10.9]{AB2006}). Note that by positive homogeneity, $\rho^\ast_Q(\nu)\in \{0,\infty\}$ (see \cite[Corollary 4.19]{FS2016}).
    If $\rho^\ast_Q(\nu)=0$, then $$\rho(\mathbf{1}_{A_n})\geq \rho^Q_E(\mathbf{1}_{A_n})\geq \nu(A_n)-\rho^\ast_Q(\nu)\geq \varepsilon, \quad n\in \N, $$ while $\rho(\mathbf{1}_{A_n})\downarrow 0$, which is absurd. Hence, $\rho^\ast_Q(\nu)=\infty$ for all $\nu\in ba^1_Q\setminus \mathfrak{P}_c(\Omega)$. Consequently, \eqref{eq:ba} is not possible, and therefore $\rho^Q_E=\rho^Q_D$.
\end{proof}

\begin{proposition}\label{prop:f_D:equal:f_E}
    Let $\rho: L^\infty_c\to \R$ be a convex risk measure and let $Q\in \mathfrak{P}_{c}(\Omega)$. Then $\rho_D^Q=\rho_E^Q$ whenever $j_Q(\mathcal{A}_\rho)$ is $\sigma(L^\infty_Q,ca_Q)$-closed. In particular, if $\rho=\rho_D$, $Q\in \mathcal{Q}_{rel}^\rho$ is supported, and  $X\in \mathcal{A}_\rho$ implies $X\mathbf{1}_{S(Q)}\in \mathcal{A}_\rho$, then $j_Q(\mathcal{A}_\rho)$ is $\sigma(L^\infty_Q,ca_Q)$-closed. 
\end{proposition}

\begin{proof} We may assume that $Q\in \mathcal{Q}_{rel}^\rho$ since otherwise $\rho_E^Q=\rho_D^Q\equiv -\infty$. 
    If $j_Q(\mathcal{A}_\rho)$ is $\sigma(L^\infty_Q,ca_Q)$-closed, one verifies that $\rho^Q_E\circ j^{-1}_Q$ is $\sigma(L^\infty_Q,ca_Q)$-lower semicontinuous. Therefore, according to \cite[Theorem 4.33]{FS2016}, $\rho_E^Q$ admits a dual representation over $ca_Q$ and $\rho_D^Q=\rho_E^Q$ by Lemma~\ref{lem:fD:fE}.
    
    Now suppose that $\rho=\rho_D$, $Q\in \mathcal{Q}_{rel}^\rho$ is supported, and for all $X\in \mathcal{A}_\rho$ we have $X\mathbf{1}_{S(Q)}\in \mathcal{A}_\rho$. Since $\rho$ admits a dual representation over $ca_c$, $\mathcal{A}_\rho$ is $\sigma(L^\infty_c,ca_c)$-closed. We show that $j_{Q}(\mathcal{A}_\rho)$ is $\sigma(L^\infty_Q,ca_Q)$-closed. To this end, let $(X^{Q}_{\alpha})_{\alpha\in I}\subseteq j_{Q}(\mathcal{A_\rho})$ converge to $X^{Q}\in L^\infty_Q$ with respect to $\sigma(L^\infty_Q,ca_Q)$. Take $X_{\alpha} \in \mathcal{A}_\rho$ and $X \in L^{\infty}_{c}$ such that $j_{Q}(X_{\alpha}) = X^{Q}_{\alpha}$, $\alpha\in I$, and $j_{Q}(X) = X^{Q}$. Then $X_{\alpha} \mathbf{1}_{S(Q)} \in \mathcal{A}_\rho$, $j_{Q}(X_{\alpha} \mathbf{1}_{S(Q)}) = X^{Q}_{\alpha}$, and $j_{Q}(X \mathbf{1}_{S(Q)}) = X^{Q}$. For any $\mu \in ca_{c}$, let $\mu_Q\in ca_Q$ be given by $\mu(A)=\mu(S(Q)\cap A)$, $A\in \mathcal{F}$. Then, as $X^{Q}_{\alpha}$ converges to $X^Q$ in $\sigma(L^\infty_Q,ca_Q)$,
    \begin{equation*}
        \int X_{\alpha} \mathbf{1}_{S(Q)} d\mu  = \int X^{Q}_{\alpha} d\mu_Q  \longrightarrow  \int X^{Q} d\mu_Q  =  \int X \mathbf{1}_{S(Q)} d\mu.
    \end{equation*}
    Since $\mu\in ca_c$ was arbitrary, we find that $X_{\alpha} \mathbf{1}_{S(Q)}$, $\alpha\in I$, converges to $X \mathbf{1}_{S(Q)}$ with respect to $\sigma(L^\infty_c, ca_c)$. Hence, as $\mathcal{A}_\rho$ is $\sigma(L^\infty_c,ca_c)$-closed, we have $X \mathbf{1}_{S(Q)} \in \mathcal{A}_\rho$ and thus $X^{Q} \in j_{Q}(\mathcal{A})$. Consequently, $j_{Q}(A_\rho)$ is $\sigma(L^\infty_Q,ca_Q)$-closed. 
\end{proof}

\begin{corollary}
    Let $\rho \colon L^{\infty}_{c} \to \R$ be a surplus-invariant convex risk measure, that is, $\rho(X) = \rho(X^+ )$, and suppose that $\rho=\rho_D$. If $Q \in \mathfrak{P}_{c}(\Omega)$ is supported, then  $\rho^Q_E=\rho^Q_D$.
\end{corollary}

We refer to \cite{GM2020} for a comprehensive discussion of surplus invariant risk measures. 

\begin{proof}
    As before we may assume that $Q\in \mathcal{Q}^\rho_{rel}$, because otherwise $\rho^Q_E=\rho^Q_D\equiv -\infty$. Note that $X\in \mathcal{A}_{\rho}$ if and only if $X^+\in \mathcal{A}_{\rho}$. As $X^+\mathbf{1}_{S(Q)}\preccurlyeq X^+$, it follows that
    \begin{equation*}
        \rho(X\mathbf{1}_{S(Q)})=\rho(X^+\mathbf{1}_{S(Q)})\leq \rho(X^+)=\rho(X),
    \end{equation*}
    and therefore $X\mathbf{1}_{S(Q)}\in \mathcal{A}_{\rho}$ whenever $X\in \mathcal{A}_{\rho}$. Now apply Proposition~\ref{prop:f_D:equal:f_E}.
\end{proof}

\begin{lemma}\label{lem:finite}
    Let $\rho \colon L^{\infty}_{c} \to \R$ be a coherent risk measure which admits a representation over finitely many $\mathfrak{P}_c(\Omega)$-constraints, that is,
    \begin{equation*}
        \rho(X) = \max_{R \in \mathcal{Q}} E_R[X], \quad X \in L^\infty_c,
    \end{equation*}
    where $\mathcal{Q}\subseteq \mathfrak{P}_c(\Omega)$ is finite.  Then $\rho^Q_E=\rho^Q_D$ for all supported $Q\in \mathfrak{P}_c(\Omega)$.
\end{lemma}
\begin{proof}
    Let $Q\in \mathcal{Q}_{rel}^\rho$ be supported. One verifies that $\overline{\rho}(R)<\infty$ if and only if $R$ lies in the convex hull of $\mathcal{Q}$. According to Lemma~\ref{lem:rhoQE:rep},
    \begin{equation*}
        \inf_{R\in  \operatorname{co}\mathcal{Q}} R(S(Q)^c)  = \min_{R\in  \mathcal{Q}}R(S(Q)^c) = 0.
    \end{equation*}
    Hence, there is $R'\in \mathcal{Q}$ such that $R'\ll Q$. Suppose there is  $R\in \mathcal{Q}$  such that $R(S(Q)^c)>0$. For any $X\in L^\infty_c$ and $X_m:=X\mathbf{1}_{S(Q)}-m\mathbf{1}_{S(Q)^c}$, $m\in \N$, we have
    \begin{equation*}
     E_{R}[X_m] =  E_{R}[X\mathbf{1}_{S(Q)}] -m R(S(Q)^c) \leq E_{R'}[X]
    \end{equation*}
    whenever
    \begin{equation*}
        m\geq \frac{E_{R}[X\mathbf{1}_{S(Q)}] - E_{R'}[X] }{R(S(Q)^c)}.
    \end{equation*}
    Consequently, as $\mathcal{Q}$ is finite and as $E_R[X_m]=E_R[X]$ for all $R\ll Q$, for $m$ large enough we obtain
    \begin{equation*}
        \rho(X_m)= \max_{R\in \mathcal{Q}} E_R[X_m] =  \max_{R\in \mathcal{Q}, R\ll Q} E_R[X],
    \end{equation*}
    and thus, by another application of Lemma~\ref{lem:rhoQE:rep},
    \begin{equation*}
        \rho^Q_E(X) = \lim_{m\to \infty} \rho(X_m) =  
         \max_{R\in \mathcal{Q}, R\ll Q} E_R[X] = \rho^Q_D(X).
    \end{equation*}
\end{proof}

Examples~\ref{ex:rhoD:rhoE} and~\ref{ex:rhoD:rhoE:2} show that $\rho_E^Q(Y)-\rho_D^Q(Y) > 0$ may happen even in a dominated framework.

\begin{example}\label{ex:rhoD:rhoE}
    Let $(\Omega,\mathcal{F}) = ((0,1), \mathcal{B}(0,1))$ and $\mathcal{P}=\{\lambda\}$, where $\lambda$ denotes the Lebesgue measure on $(\Omega,\mathcal{F})$. Set $A_n:=(0, \frac{1}{n+1})$ and $B:=[\frac{1}{2}, 1)$, and define probability measures $P_{n}$ on $(\Omega, \mathcal{F})$ by 
    \begin{equation*}
        P_{n}(\, \cdot \,) := \frac{n-1}{n} \lambda(\, \cdot \mid A_{n}) + \frac{1}{n} \lambda(\, \cdot \mid B), \quad n \in \N.
    \end{equation*}
    Consider the coherent risk measure
    \begin{equation*}
        \rho(X) := \sup_{n \in \N} E_{P_{n}}[X], \quad X \in L^{\infty}_{\lambda}.
    \end{equation*}
    Note that Lemma~\ref{lem:finite} does not apply, since $\rho$ is given by infinitely many constraints. As $\rho$ admits a dual representation over $ca_{\lambda}$ by definition, it follows from Theorem~\ref{thm:dual:f} that
    \begin{equation*}
        \rho = \rho_{D} = \rho_{E}.
    \end{equation*}
     However, as we will see in the following, there exists $Q \in \mathfrak{P}_{\lambda}(\Omega)$ such that $\rho^{Q}_{D} \neq \rho^{Q}_{E}$.
    
    Set $Y := [y]_{\lambda}$, where
    \begin{equation*}
        y(\omega) :=
        \begin{cases}
            -1, &\text{if } \omega \in B, \\
            2 \omega, &\text{otherwise.}
        \end{cases}
    \end{equation*}
    Then $Y \in L^{\infty}_{\lambda}$ and 
    \begin{equation*}
        E_{P_{n}}[Y] =  \bigg(\frac{(n-1)(n+1)}{n} \int_{0}^{\frac{1}{n+1}} 2 \omega \, d \omega - \frac{1}{n}\bigg) =  -\frac{2}{(n+1)n} \leq 0.
    \end{equation*}
    Hence, $\rho(Y) = 0$ and $Y \in \mathcal{A}_{\rho}$. Consider any probability measure $R \in \mathfrak{P}_\lambda(\Omega)$ such that $R(B) = 0$. As $R(Y > 0) =  1$, it follows that $E_{R}[Y] > 0$. Consequently, recalling that $tY\in \mathcal{A}_{\rho}$ for all $t>0$ by positive homogeneity, we have
    \begin{equation*}
        \overline{\rho}(R) = \sup_{X \in \mathcal{A}_{\rho}} E_{R}[X] \geq \sup_{t > 0} E_{R}[tY] = \infty.
    \end{equation*}
    Consider the supported probability measure $Q=\lambda(\, \cdot\mid (0,\frac{1}{2}))$. Indeed $S(Q)=(0,\frac{1}{2})=B^c$. Any $R \ll Q$ has to satisfy $R(B) = R(S(Q)^{c}) = 0$ and therefore $\overline{\rho}(R) = \infty$. It follows that $\rho^{Q}_{D} \equiv -\infty$.
    
    As regards $\rho^{Q}_{E}$, we have 
    \begin{equation*}
     \inf_{\overline{\rho}(R) < \infty} R\big(S(Q)^{c}\big) = \inf_{\overline{\rho}(R) < \infty} R(B) \leq \inf_{n \in \N} P_{n}(B) = \inf_{n\in \N} \frac{1}{n} = 0.
    \end{equation*}
    Hence, by Lemma~\ref{lem:rhoQE:rep} $\rho^{Q}_{E}$ is a coherent risk measure. Consequently,  $\rho^{Q}_{E}(X) > \rho^{Q}_{D}(X)$ for all $X\in L^\infty_c$ and the size of the localization bubble is infinite. \hfill $\diamond$ 
\end{example}

\begin{example}\label{ex:rhoD:rhoE:2}
    Recall Example~\ref{ex:rhoD:rhoE}. This time define
    \begin{equation*}
        \kappa(X):= \rho(X) \vee E_Q[X], \quad X\in L^\infty_\lambda,
    \end{equation*}
    which is another coherent risk measure. Suppose that $R\ll Q$ but $R\neq Q$. Pick a version $\varphi$ of the density $\frac{dR}{dQ}$ such that $\varphi\chi_{[\frac12,1)}=0$. As $R\neq Q$, we must have $Q(\varphi>1)>0$ and $Q(\varphi<1)>0$. Moreover, $A_n\downarrow \emptyset$ and thus 
    \begin{equation*}
        \lim_{n\to \infty}Q(\{\varphi>1\}\cap A_n)=0 \quad \text{and} \quad \lim_{n\to \infty} Q(\{\varphi<1\} \cap A_n)=0.
    \end{equation*}
    Choose $m$ large enough such that
    \begin{equation*}
        Q(\{\varphi>1\} \cap A_m) < Q(\{\varphi>1\}) \quad \text{and} \quad Q(\{\varphi < 1\} \cap A_m) < Q(\{\varphi < 1\}).
    \end{equation*}
    Then there is
    \begin{equation*}
        D\subseteq \{\varphi > 1\} \setminus A_m  \quad \text{and} \quad D' \subseteq (\{\varphi < 1\} \setminus A_m)\cap (0, \tfrac12)
    \end{equation*}
    such that $Q(D) = Q(D') > 0$. Note that $\lambda(D\mid A_n)=\lambda(D'\mid A_n)=0$ for all $n\geq m$. Consider $Z := [z]_\lambda$, where
    \begin{equation*}
        z := \chi_{D} - \chi_{D'} - m \chi_B.
    \end{equation*}
    Then, $E_Q[Z] = 0$ and $E_{P_n}[Z] = -\frac{m}{n} \leq 0$ for $n \geq m$. For $n < m$, we have
    \begin{equation*}
        E_{P_n}[Z] \leq \frac{(n-1)(n+1)}{n} \int_0^{\frac{1}{n+1}} 1 \, d\omega -\frac{m}{n} = \frac{n-1-m}{n} \leq 0.
    \end{equation*}
    Therefore, $Z \in \mathcal{A}_{\kappa}$. However,
    \begin{equation*}
        E_R[Z] = E_Q[\varphi Z] > Q(D) -Q(D') = 0
    \end{equation*}
    which implies that $\overline{\kappa}(R)\geq \sup_{t>0}E_R[tZ] = \infty$. Clearly, $\overline{\kappa}(Q)=0$, and therefore $\kappa^Q_D(\,\cdot\,) = E_Q[\,\cdot\,]$. 
    
    Regarding $\kappa^Q_E$, as in Example~\ref{ex:rhoD:rhoE}, we conclude that $\inf_{\overline{\kappa}(R) < \infty} R\big(S(Q)^{c}\big) = 0$, which implies that $\kappa^Q_E$ is a coherent risk measure (see Lemma~\ref{lem:rhoQE:rep}). Hence, both $\kappa^Q_D$ and $\kappa^Q_E$ are coherent risk measures. Consider $W := [-y]_\lambda$, where $y$ was given in Example~\ref{ex:rhoD:rhoE}. Then $\kappa^Q_D(W) = E_Q[W] < 0$. For any $m \in \N$,
    \begin{equation*}
        E_{P_{n}}[W \mathbf{1}_{(0,\frac12)} - m \mathbf{1}_{B}] = -\frac{(n-1)}{n(n+1)}  - \frac{m}{n} \leq 0,
    \end{equation*}
    and thus, as $E_Q[W \mathbf{1}_{(0,\frac12)} - m \mathbf{1}_{B}] = E_Q[W] < 0$,
    \begin{equation*}
        \kappa(W \mathbf{1}_{(0,\frac12)} - m \mathbf{1}_{B}) = \sup_{n \in \N} \bigg(-\frac{(n-1)}{n(n+1)}  - \frac{m}{n}\bigg) = 0.
    \end{equation*}
    By applying Lemma~\ref{lem:rhoQE:rep}, we therefore have that
    \begin{equation*}
        \kappa^Q_E(W) = \lim_{m \to \infty} \kappa(W \mathbf{1}_{(0,\frac12)} - m \mathbf{1}_{B}) = 0.
    \end{equation*}
    Hence, there is a finite localization bubble $\kappa^Q_E(W) - \kappa^Q_D(W) > 0$. \hfill $\diamond$
\end{example}

\subsection{Arbitrage and Superhedging in One Period} \label{sec:FTAP}

Again, we let $\mathcal{X} = L^{\infty}_{c}$ with $ca_{c}(\mathcal{X}) = ca_{c}$. A \textit{one-period market model} of dimension $d \in \N$ is given by a (discounted) \textit{stock price process} $S=(S_0, S_1) \in \mathcal{S}$, where $S_{0} \in \R^{d}$ and $S_{1}$ is a bounded $d$-dimensional random vector. All coordinates $S^i_t$, $i = 1, \ldots, d$, of $S_t$, $t = 0, 1$, are assumed to be non-negative and $(S^i_0,S^i_1)$ will be called the $i$-th asset of $S$, $i=1,\ldots, d$.  The space of all such market models is denoted by 
\begin{equation*}
    \mathcal{S} := \{S=(S_{0}, S_{1}) \mid S \text{ is a one-period market model of dimension } d, \ d \in \N\}.
\end{equation*}
Let $d(S)$ denote the dimension of $S\in \mathcal{S}$, and we write $S \lhd S'$  to indicate that $S\in \mathcal{S}$ is a \textit{submarket} of $S'\in \mathcal{S}$, that is, for all $i\in \{1,\ldots, d(S)\}$ there is $j\in \{1, \ldots, d(S')\}$ such that $(S_0^i,S_1^i)=(S_0^j,S_1^j)$. Moreover, for $S=(S_0,S_1)\in \mathcal{S}$, let
\begin{equation*}
    \Delta S:=(\Delta S^1,\ldots \Delta S^{d(S)}) := S_{1} - S_{0} = (S^1_{1} - S^1_{0},\ldots,  S^{d(S)}_{1} - S^{d(S)}_{0})
\end{equation*}
and recall that a probability measure $Q\in \mathfrak{P}(\Omega)$ is called a \textit{martingale measure} for the market model $S$ if, for each $i=1,\ldots, d(S)$, $S^i_1$ is $Q$-integrable and
\begin{equation*}
    E_{Q}[\Delta S] := (E_{Q}[\Delta S^1],\ldots, E_{Q}[\Delta S^{d(S)}]) = 0.
\end{equation*}
The space of \textit{(investment) strategies} for a given $d$-dimensional market $S \in \mathcal{S}$ is given by $\mathcal{H}(S):=\R^d$.

In the following, for $Q \in \mathfrak{P}_c(\Omega)$, we write $Q \lll \mathcal{P}$ if there exists $P \in \mathcal{P}$ such that $Q \ll P$. We will consider the following set-valued maps:
\begin{itemize}
    \item $\mathfrak{M} \colon \mathcal{S} \twoheadrightarrow \mathcal{P}_{c}(\Omega)$ defined by
    \begin{equation*}
        \mathfrak{M}(S) := \{Q \in \mathfrak{P}_{c}(\Omega) \mid Q\, \mbox{is a  martingale measure for}\, S\}, 
    \end{equation*}

    \item $\mathfrak{NA}\colon \mathcal{S} \twoheadrightarrow \mathfrak{P}_{c}(\Omega)$ defined by 
    \begin{equation*}
        \mathfrak{NA}(S) := \{R \in \mathfrak{P}_{c}(\Omega) \mid \exists Q \in \mathfrak{M}(S) \colon  R \approx Q\},
    \end{equation*}
    
    \item $\mathfrak{M}_{\ll} \colon \mathcal{S} \twoheadrightarrow \mathfrak{P}_{c}(\Omega)$ defined by
    \begin{equation*}
        \mathfrak{M}_{\ll}(S) := \{Q \in \mathfrak{P}_{c}(\Omega) \mid Q \in \mathfrak{M}(S) \text{ and } Q \lll \mathcal{P}\},
    \end{equation*}
    
    \item $\mathfrak{M}_{\approx} \colon \mathcal{S} \twoheadrightarrow \mathfrak{P}_{c}(\Omega)$ defined by
    \begin{equation*}
        \mathfrak{M}_{\approx}(S) := \{Q \in \mathfrak{P}_{c}(\Omega) \mid Q \in \mathfrak{M}(S) \text{ and } \exists P \in \mathcal{P} \colon Q \approx P\},
    \end{equation*}

    \item $\mathfrak{M}^{Q}_{\approx} \colon \mathcal{S} \twoheadrightarrow \mathfrak{P}_{c}(\Omega)$, where $Q \in \mathfrak{P}_{c}(\Omega)$, defined by
    \begin{equation*}
        \mathfrak{M}^{Q}_{\approx}(S) := \{R \in \mathfrak{P}_{c}(\Omega) \mid R \in \mathfrak{M}(S) \text{ and } R \approx Q \}.
    \end{equation*} 
\end{itemize}
Clearly, $\mathfrak{M}_{\approx}(S) \subseteq \mathfrak{M}_{\ll}(S) \subseteq \mathfrak{M}(S) \subseteq \mathfrak{NA}(S)$ for any $S \in \mathcal{S}$. Following the standard approach to no-arbitrage in robust models (see, e.g., \cite[Definition 1.1]{BN2015}), we say that the \textit{no-arbitrage condition} NA($\mathcal{P}, S$) holds for the market $S \in \mathcal{S}$ if, for all $H \in \mathcal{H}(S)$,
\begin{equation*}
    H \Delta S \geq 0 \ \mathcal{P}\text{-q.s.} \quad \text{implies} \quad H \Delta S = 0 \ \mathcal{P}\text{-q.s.}
\end{equation*}
Note that $H \Delta S$ is the usual ($\omega$-wise) Eulidean scalar product of the vectors $H$ and $\Delta S$.
Similarly, for any $Q\in \mathfrak{P}_c(\Omega)$, $S \in \mathcal{S}$ satisfies NA($Q,S$) if, for all $H \in \mathcal{H}(S)$,
\begin{equation*}
    H \Delta S \geq 0 \ Q\text{-a.s.} \quad \text{implies} \quad H \Delta S = 0 \ Q\text{-a.s.} 
\end{equation*}
For any given $S\in \mathcal{S}$, we define the \textit{$\mathcal{P}$-superhedging functional} $\pi(\, \cdot \mid S) \colon L^\infty_c \to \R\cup\{-\infty\}$ by
\begin{equation}\label{eq:SHFc}
    \pi(X \mid S) := \inf \{r \in \R \mid \exists H \in \mathcal{H}(S) \colon X \preccurlyeq r + [H \Delta S]_c \}.
\end{equation}
Considering the strategy $H = 0$, we see that $\pi(\, \cdot \mid S)$ indeed takes values in $\R \cup \{-\infty\}$ since $\pi(X \mid S) \leq \|X\|_{c,\infty}$ for all $X\in L^\infty_c$. Finally, for $Q \in \mathfrak{P}_{c}(\Omega)$, 
\begin{equation*}
    \pi^{Q}(X \mid S) := \inf \{r \in \R \mid \exists H \in \mathcal{H}(S) \colon X \leq_Q r + [H \Delta S]_c \}, \quad X \in L^\infty_{c}, 
\end{equation*}
is the \textit{$Q$-superhedging functional} for a given $S \in \mathcal{S}$. Note that both $\pi(\, \cdot \mid S)$ and $\pi^{Q}(\, \cdot \mid S)$ are coherent risk measures on $L^\infty_c$, provided they are $\R$-valued, which is equivalent to $\pi(0 \mid S)=0$ and $\pi^{Q}(0\mid S)=0$, respectively (by the same reasoning as in the proof of Lemma~\ref{lem:rhoE:rm}).

\begin{lemma} \label{lemm:dual:rep:pi}
    Let $S\in \mathcal{S}$. Then
    \begin{equation*}
        \{R\in \mathfrak{P}_c(\Omega) \mid \overline{\pi}(R\mid S)<\infty\} = \{R\in \mathfrak{P}_c(\Omega)\mid \overline{\pi}(R\mid S)=0\} = \mathfrak{M}(S).
    \end{equation*}
\end{lemma}
\begin{proof}
    We may assume that $\pi(0 \mid S)=0$, so that $\pi(\, \cdot \mid S)$ is a coherent risk measure on $L^\infty_c$, because otherwise, $\pi(\, \cdot \mid S) \equiv -\infty$ and thus $\overline{\pi}(\, \cdot\mid S) \equiv \infty$. The first equality is a well-known consequence of coherence of $\pi(\, \cdot\mid S)$ (see \cite[Corollary~4.19]{FS2016}). Let $R \in \mathfrak{M}(S)$. Then, by \eqref{eq:SHFc}, $E_R[X]\leq \pi(X\mid S)$ for all $X\in L^\infty_c$. Lemma~\ref{lem:rhoD} implies $\mathfrak{M}(S)\subseteq \{R\in \mathfrak{P}_c(\Omega)\mid \overline{\pi}(R\mid S)=0\}$. 
    
    Conversely, let $R\in \mathfrak{P}_c(\Omega)$ satisfy $\overline{\pi}(R\mid S)=0$, that is, $E_R[X]\leq 0$ for all $X\in L^\infty_c$ such that $\pi(X \mid S)\leq 0$ (see Lemma~\ref{lem:rhoD}). Since both $\pi([\Delta S^i]_c \mid S)\leq 0$ and $\pi(-[\Delta S^i]_c \mid S)\leq 0$ for all $i=1, \ldots, d(S)$, we conclude that $R\in \mathfrak{M}(S)$.
\end{proof}

\begin{proposition} \label{prop:piQ:piQE:pi:ineq}
    Let $S\in \mathcal{S}$ and $Q \in \mathfrak{P}_{c}(\Omega)$. Then
    \begin{equation*}
        \pi^Q_D(\, \cdot \mid S)\leq \pi^Q(\, \cdot \mid S)\leq \pi^{Q}_{E}(\, \cdot \mid S) \leq \pi(\, \cdot \mid S).
    \end{equation*}
    If $\operatorname{NA}(Q,S)$ holds, then $\pi^Q_D(\, \cdot \mid S)= \pi^Q(\, \cdot \mid S)$. If $Q$ is supported, we have $ \pi^Q(\, \cdot \mid S)= \pi^{Q}_{E}(\, \cdot \mid S)$. 
\end{proposition}
\begin{proof}
    If $X \preccurlyeq r + [H \Delta S]_c$, then $X \leq_Q r + [H \Delta S]_c$. Therefore, $\pi^Q(X \mid S) \leq \pi(X \mid S)$. Thus, as $\pi^{Q}(\, \cdot \mid S)$ is $Q$-consistent, $\pi^Q(\, \cdot \mid S) \leq \pi^{Q}_{E}(\, \cdot \mid S) \leq \pi(\, \cdot \mid S)$ follows from Lemma~\ref{lem:aux:1}. Recall Lemma~\ref{lemm:dual:rep:pi}. Let $R \ll Q$ such that $\overline{\pi}(R \mid S) = 0$. Then, for each $X \in L^\infty_c$, $H \in \mathcal{H}(S)$, and $r \in \R$ such that $X \leq_Q r + [H \Delta S]_c$, we have $E_R[X] \leq r$, because $R \in \mathfrak{M}(S)$. Therefore, $\pi^Q(\, \cdot \mid S) \geq \pi^Q_D(\, \cdot \mid S)$.
    
    For any $Q \in \mathfrak{P}_c(\Omega)$ such that $\operatorname{NA}(Q, S)$ holds, the superhedging duality for dominated models (see \cite[Theorem~1.32]{FS2016}) and Lemma~\ref{lem:fD:fE} imply that $\pi^Q(\, \cdot \mid S) \leq \pi^Q_D(\, \cdot \mid S)$, because $\pi^Q(\, \cdot \mid S)$ admits a dual representation over $ca_Q$.  
    
    Suppose that $Q$ is supported and let $X \leq_Q r + [H \Delta S]_c$. Define
    \begin{equation*}
        Y := X \mathbf{1}_{S(Q)} + (r + [H \Delta S]_c) \mathbf{1}_{S(Q)^c}.
    \end{equation*}
    Then $j_Q(Y) = j_Q(X)$ and $Y \preccurlyeq r + [H \Delta S]_c$. Therefore, $\pi_E^Q(X\mid S) \leq \pi(Y\mid S) \leq r$, and hence $\pi_E^Q(X\mid S)\leq \pi^Q(X\mid S)$.
\end{proof}

In view of the discussion in Section~\ref{sec:rm}, a natural question is whether the superhedging functionals admit localization bubbles. The answer is no, at least if $Q$ is supported and $\operatorname{NA}(Q,S)$ is satisfied, because then Proposition~\ref{prop:piQ:piQE:pi:ineq} shows that $\pi^Q_D(\, \cdot \mid S)= \pi^Q(\, \cdot \mid S)= \pi^{Q}_{E}(\, \cdot\mid S)$.

The following lemma is a key observation for the results that follow.

\begin{lemma}\label{lem:locally:consistent}
    For any market model $S \in \mathcal{S}$ such that $\operatorname{NA}(\mathcal{P},S)$ is satisfied, there exists $P \in \operatorname{co}(\mathcal{P})$, where $\operatorname{co}(\mathcal{P})$ is the convex hull of $\mathcal{P}$, such that $\operatorname{NA}(P,S)$ holds.
\end{lemma}
For the sake of completeness, we provide the proof below, even though the result follows from, e.g., \cite[Lemma 2.7]{BZ2015}:
\begin{proof}
    Fix $S \in \mathcal{S}$ with dimension $d := d(S)$ such that NA($\mathcal{P}$) holds. Let
    \begin{equation*}
        N(\mathcal{P}) := \{H \in \R^{d} \mid [H \Delta S]_c = 0 \}\quad \mbox{and} \quad N(\mathcal{P})^{\perp} = \{H \in \R^{d} \mid \forall \widetilde{H} \in N(\mathcal{P}) \colon \widetilde{H} \cdot H = 0\},
    \end{equation*}
    where, for the sake of better readability, the Euclidean scalar product between the vectors $\widetilde{H}$  and $H$ is denoted by $\widetilde{H} \cdot H$. 
     
    First suppose that $N(\mathcal{P})^\perp=\{0\}$. Then, for all $H\in \R^d$, $H\Delta S=0$ $\mathcal{P}$-q.s.\ and thus $H\Delta S=0$ $P$-a.s.\ for any $P\in \mathfrak{P}_c(\Omega)$. In particular, NA($P,S$) holds for any $P\in \mathcal{P}$. From now on assume that $N(\mathcal{P})^\perp\neq \{0\}$. Then
    \begin{equation*}
        \mathbb{H} := \{H \in N(\mathcal{P})^{\perp} \mid \lVert H \rVert = 1\}\neq \emptyset,
    \end{equation*}
    where $\lVert \, \cdot \, \rVert$ denotes the Euclidean norm. For any $H \in \mathbb{H}$, by NA($\mathcal{P},S$), there exists $P_{H} \in \mathcal{P}$ such that $P_{H}(H \Delta S < 0) > 0$. Moreover, there exists $\varepsilon_{H} > 0$ such that for each $H' \in B(H, \varepsilon_{H}):=\{\widetilde{H}\mid \|\widetilde{H}-H\|< \varepsilon_{H}\}$, we have
    \begin{equation*}
        P_{H}(H' \Delta S < 0) > 0.
    \end{equation*}
    Indeed, there exists $\delta > 0$ such that $P_{H}(H \Delta S < -\delta) > 0$, and, by boundedness of $S_1$, there is $M > 0$ such that $\lVert \Delta S \rVert < M$. Let  $\varepsilon_{H} := \delta / M$. For any $H' \in B(H, \varepsilon_{H})$,
    \begin{equation*}
        H' \Delta S - H \Delta S \leq \lvert H' \Delta S - H \Delta S \rvert \leq \lVert H' - H \rVert \cdot \lVert \Delta S \rVert \leq \varepsilon_{H} \cdot M = \delta.
    \end{equation*}
    Thus, $P_{H}(H' \Delta S < 0)\geq P_{H}(H \Delta S < -\delta)$. Note that $\mathbb{H} \subseteq \bigcup_{H \in \mathbb{H}} B(H, \varepsilon_{H})$. As $\mathbb{H}$ is compact, there exists a finite subcover of $\mathbb{H}$ given by $H_1, \ldots, H_n\in \mathbb{H}$, i.e., $\mathbb{H} \subseteq \bigcup_{i = 1}^{n} B(H_{i}, \varepsilon_{H_{i}})$. Let $P := \frac1n\sum_{i = 1}^{n}  P_{H_{i}}$. Then $P \in \operatorname{co}(\mathcal{P})$ and $P(H \Delta S < 0) > 0$ for all $H \in \mathbb{H}$. We now verify that  NA($P, S$) holds. To this end, let $H \in \R^d$ and decompose $H = H^0 + H^\perp$, where $H^0 \in N(\mathcal{P})$ and $H^\perp \in N(\mathcal{P})^\perp$. Since $H^0 \Delta S = 0$ $\mathcal{P}$-q.s., it follows that $H^0\Delta S=0$ $P$-a.s. If $H^\perp=0$, then $H\Delta S=H^0\Delta S=0$ $P$-a.s. Otherwise, if $H^\perp\neq 0$, then $H^\perp/\|H^\perp\|\in \mathbb{H}$, and therefore
    \begin{equation*}
        P(H\Delta S < 0) = P(H^\perp \Delta S < 0)=P\bigg(\frac{H^\perp}{\|H^{\perp}\|}\Delta S < 0\bigg) > 0.
    \end{equation*}
    Thus, NA($P,S$) is satisfied. 
\end{proof}

As we will see below, the following class of set-valued maps $\mathfrak{Q} \colon \mathcal{S} \twoheadrightarrow \mathfrak{P}_{c}(\Omega)$ is closely related to the $\mathcal{P}$-sensitivity of the superhedging functional. 

\begin{definition}
    A set-valued map $\mathfrak{Q} \colon \mathcal{S} \twoheadrightarrow \mathfrak{P}_{c}(\Omega)$ is called 
    \begin{enumerate}
        \item[(i)] \textit{$\lhd$-monotone} if $\mathfrak{Q}(S') \subseteq \mathfrak{Q}(S)$ for any $S, S' \in \mathcal{S}$ such that $S \lhd S'$,
        
        \item[(ii)] \textit{weakly $\operatorname{NA}$-preserving} if for any $S \in \mathcal{S}$ that satisfies $\operatorname{NA}(\mathcal{P},S)$, there exists $Q \in \mathfrak{Q}(S)$ such that $\operatorname{NA}(Q,S)$ holds,

        \item[(iii)] \textit{strongly $\operatorname{NA}$-preserving} if for any $S \in \mathcal{S}$ that satisfies $\operatorname{NA}(\mathcal{P},S)$, $\mathfrak{Q}(S) \neq \emptyset$ and $\operatorname{NA}(Q,S)$ holds for all $Q \in \mathfrak{Q}(S)$.
    \end{enumerate}
\end{definition}

\begin{lemma} \label{lem:NAP:NAQ3}
    The maps $\mathfrak{Q}(S) \equiv \mathcal{P}$ for $S \in \mathcal{S}$, $\mathfrak{M}$, $\mathfrak{NA}$, $\mathfrak{M}_{\ll}$, and $\mathfrak{M}_{\approx}$ are all $\lhd$-monotone. Moreover, $\mathfrak{M}$ and $\mathfrak{NA}$ are strongly $\operatorname{NA}$-preserving. If $\mathcal{P}$ is convex, then $\mathfrak{Q}(S) \equiv \mathcal{P}$, $S \in \mathcal{S}$, is weakly $\operatorname{NA}$-preserving, while $\mathfrak{M}_{\ll}$ and $\mathfrak{M}_{\approx}$ are strongly $\operatorname{NA}$-preserving.
\end{lemma}
\begin{proof}
    $\lhd$-monotonicity follows directly from the respective definitions of the maps. 
    Let $S\in \mathcal{S}$ satisfy $\operatorname{NA}(\mathcal{P},S)$. Then, by Lemma~\ref{lem:locally:consistent}, there exists $P \in \operatorname{co}(\mathcal{P})$ such that NA($P,S$). Thus, by the Fundamental Theorem of Asset Pricing (see, e.g., \cite[Theorem~1.7]{FS2016}), there is a martingale measure $Q\approx P$. In particular, $\mathfrak{M}(S) \neq \emptyset$ and  NA($R,S$) is satisfied for any $R \in \mathfrak{NA}(S)$. Recalling that $\mathfrak{M}(S) \subseteq \mathfrak{NA}(S)$, we conclude that $\mathfrak{M}$ and $\mathfrak{NA}$ are strongly $\operatorname{NA}$-preserving. If $\mathcal{P} = \operatorname{co}(\mathcal{P})$, then the previous arguments also prove the remaining assertions.
\end{proof}

\begin{theorem} \label{thm:hedge:Q}
    Let $\mathfrak{Q} \colon \mathcal{S} \twoheadrightarrow \mathfrak{P}_{c}(\Omega)$ be $\lhd$-monotone and weakly $\operatorname{NA}$-preserving. Suppose that $S \in \mathcal{S}$ satisfies $\operatorname{NA}(\mathcal{P},S)$. Then $\pi(\, \cdot \mid S)$ is $\mathcal{P}$-sensitive with reduction set $\mathfrak{Q}(S)$ and $(\pi^{Q}(\, \cdot \mid S))_{Q \in \mathfrak{Q}(S)}$ is a $\mathfrak{Q}(S)$-localization of  $\pi(\, \cdot \mid S)$. In particular, 
    \begin{equation}\label{eq:pi:piQ}
        \pi(X \mid S) = \sup_{Q \in \mathfrak{Q}(S)} \pi^{Q}_{E}(X \mid S) = \sup_{Q \in \mathfrak{Q}(S)} \pi^{Q}(X \mid S) =: \pi^{\mathfrak{Q}(S)}(X \mid S).
    \end{equation}
\end{theorem}
\begin{proof}  
    Suppose that $S \in \mathcal{S}$ satisfies NA($\mathcal{P},S$). By Proposition~\ref{prop:piQ:piQE:pi:ineq}, we have
    \begin{equation}\label{eq:whatever}
        \pi(\, \cdot \mid S) \geq \pi^Q_E(\, \cdot \mid S) \geq \pi^{Q}(\, \cdot \mid S)
    \end{equation}
    for all $Q\in \mathfrak{Q}(S)$. Let $X \in L^{\infty}_{c}$. If $\pi(X\mid S) = -\infty$, then \eqref{eq:pi:piQ} follows. Assume that $\pi(X) > -\infty$ and let
    \begin{equation*}
        \widetilde{\pi}(Y \mid S) := \sup \{x \in \R \mid \exists H \in \mathcal{H}(S) \colon Y \succcurlyeq x + [H \Delta S]_c \}, \quad Y\in L^\infty_c,
    \end{equation*}
    be the subhedging price for the market model $S$. One verifies that NA($\mathcal{P},S$) implies $\widetilde{\pi}(\, \cdot \mid S) \leq \pi(\, \cdot \mid S)$. If $\widetilde{\pi}(X \mid S) < \pi(X \mid S)$, let $x \in (\widetilde{\pi}(X \mid S), \pi(X \mid S))$ and consider the extended market $S_{X} = ((S_0, x), (S_1, X))$. Let $H \in \R^d$, $h \in \R$ satisfy
    \begin{equation*}
        [H \Delta S]_{c} + h (X - x) \succcurlyeq 0.
    \end{equation*}
    If $h = 0$, NA($\mathcal{P},S$) implies that $[H \Delta S]_c = 0$. If $h \neq 0$ then either
    \begin{equation*}
       - \bigg[\frac{H}{h} \Delta S\bigg]_c + x \succcurlyeq X \quad \mbox{or} \quad 
        -\bigg[\frac{H}{h} \Delta S\bigg]_c + x \preccurlyeq X.
    \end{equation*}
    Hence, either $x\geq \pi(X \mid S)$ or $x\leq \widetilde{\pi}(X \mid S)$, which is absurd. Thus, NA($\mathcal{P},S_X$) is satisfied. As $\mathfrak{Q}$ is weakly $\operatorname{NA}$-preserving, there exists $Q \in \mathfrak{Q}(S_{X})$ such that NA($Q,S_X$) holds.  Hence, by the Fundamental Theorem of Asset Pricing for dominated models (see \cite[Theorem~1.7]{FS2016}), there is a martingale measure $R \in \mathfrak{P}_c(\Omega)$ for $S$ which is equivalent to $Q$ and satisfies $E_{R}[X] = x$. The superhedging duality for dominated models (see \cite[Theorem~1.32]{FS2016}) implies $\pi^{Q}(X \mid S) \geq E_{R}[X] = x$. By $\lhd$-monotonicity, we have $Q \in \mathfrak{Q}(S)$. Therefore, as $x \in (\widetilde{\pi}(X \mid S), \pi(X \mid S))$ was arbitrary, it follows that
    \begin{equation*}
        \pi^{\mathfrak{Q}(S)}(X \mid S) = \sup_{Q \in \mathfrak{Q}(S)} \pi^{Q}(X \mid S) \geq \pi(X \mid S),
    \end{equation*}
    which, in conjunction with \eqref{eq:whatever}, implies \eqref{eq:pi:piQ}. 

    Finally, suppose that $\widetilde{\pi}(X \mid S) = \pi(X \mid S)$. For every $\varepsilon > 0$ there are $H, \widetilde{H} \in \R^d$ such that
    \begin{equation}\label{eq:help:pi:tildepi}
        \pi(X \mid S)+ \varepsilon + [H \Delta S]_c \succcurlyeq X \succcurlyeq \widetilde{\pi}(X \mid S) - \varepsilon + [\widetilde{H} \Delta S]_c = \pi(X \mid S) - \varepsilon + [\widetilde{H} \Delta S]_c.
    \end{equation}
    As $\mathfrak{Q}$ is weakly $\operatorname{NA}$-preserving, there exists $Q \in \mathfrak{Q}(S)$ such that NA($Q,S$) holds. Consider any martingale measure $R \in \mathfrak{P}_c(\Omega)$ which is equivalent to $Q$. Taking expectations under $R$ in \eqref{eq:help:pi:tildepi} yields
    \begin{equation*}
        \pi(X \mid S) + \varepsilon \geq E_R[X] \geq \pi(X \mid S) - \varepsilon.
    \end{equation*}
    Since the latter is true for any $\varepsilon>0$, we conclude that $E_R[X] = \pi(X \mid S)$. Again, the superhedging duality in dominated models implies $\pi^{Q}(X \mid S) \geq E_R[X] = \pi(X \mid S)$, and thus
    \begin{equation*}
        \pi^{\mathfrak{Q}(S)}(X \mid S) = \sup_{Q \in \mathfrak{Q}(S)} \pi^{Q}(X \mid S) \geq \pi(X \mid S),
    \end{equation*}
    which, in conjunction with \eqref{eq:whatever}, implies \eqref{eq:pi:piQ}. 
\end{proof}

\begin{corollary}
    Suppose that $S\in \mathcal{S}$ satisfies $\operatorname{NA}(\mathcal{P},S)$. Then $\pi(\, \cdot\mid S)=\pi_E(\, \cdot\mid S)=\pi_D(\, \cdot\mid S)$. 
\end{corollary}
\begin{proof}
    Letting $\mathfrak{Q}=\mathfrak{NA}$ and recalling Lemma~\ref{lem:NAP:NAQ3}, we have $\pi(\, \cdot \mid S)=\pi^{\mathfrak{NA}(S)}(\, \cdot \mid S)$ by Theorem~\ref{thm:hedge:Q}. The superhedging duality for dominated models (see \cite[Theorem~1.32]{FS2016}) implies that $\pi(\, \cdot \mid S)$ admits a dual representation over $ca_c$, since each $\pi^Q(\, \cdot \mid S)$ does so for every $Q \in \mathfrak{NA}(S)$. Hence, $\pi(\, \cdot\mid S)=\pi_E(\, \cdot\mid S)=\pi_D(\, \cdot\mid S)$ according to Theorem~\ref{thm:dual:f}. 
\end{proof} 

\begin{corollary}\label{cor:hedge:dual:Q}
    Let $\mathfrak{Q} \colon \mathcal{S} \twoheadrightarrow \mathfrak{P}_{c}(\Omega)$ be $\lhd$-monotone and strongly $\operatorname{NA}$-preserving. Then, for every $S \in \mathcal{S}$ such that $\operatorname{NA}(\mathcal{P}, S)$ is satisfied, we have
    \begin{equation*}
        \pi(X \mid S) = \sup_{Q \in \mathfrak{Q}(S)} \sup_{R \in \mathfrak{M}^{Q}_{\approx}(S)} E_{R}[X].
    \end{equation*}
\end{corollary}
\begin{proof}
Theorem~\ref{thm:hedge:Q} and the superhedging duality for dominated models.
\end{proof}

Lemma~\ref{lem:C:clsd} and Corollary~\ref{cor:opt:strat} below can also be found in \cite[Theorems 2.2 and 2.3]{BN2015}. For the sake of completeness, we provide the proofs for the one-period case with deterministic trading strategies based on the proof of \cite[Theorem~1.32]{FS2016}.

\begin{lemma} \label{lem:C:clsd}
    Suppose that $S\in \mathcal{S}$ satisfies $\operatorname{NA}(\mathcal{P},S)$. Then
    \begin{equation*}
        \mathcal{C}(S) := \{X \in L^\infty_c \mid \exists H \in \R^d \colon X \preccurlyeq [H \Delta S]_{c}\}.
    \end{equation*}
    is closed under $\mathcal{P}$-q.s.\ convergence.
\end{lemma}
\begin{proof}
    $X \in \mathcal{C}(S)$ if and only if $X = [H \Delta S]_{c} - Y$ for some $H \in \R^{d}$ and $Y \in L^{0}_{c+}$. Let $(X_{n})_{n \in \N} \subseteq \mathcal{C}(S)$ be a sequence which converges $\mathcal{P}$-q.s.\ to a random variable $X \in L^{\infty}_{c}$. Then, for any $n \in \N$, $X_{n} = [H_{n} \Delta S]_{c} - Y_{n}$, where $H_n \in \R^{d}$ and $Y_n \in L^{0}_{c+}$, and we will show that there exists $H \in \R^{d}$ and $Y \in L^{0}_{c+}$ such that $X = [H \Delta S]_{c} - Y$. To this end, we may assume that the market $S$ is non-redundant in the sense that $[H \Delta S]_{c}=0$ implies $H=0$. Otherwise, there is a component $i\in \{1,\ldots, d\}$ and constants $a^j \in \R$, $j \in \{1, \ldots, d\} \setminus \{i\}$, such that $S^i_1-S_0^i=\sum_{j\neq i} a^j(S^j_1-S_0^j)$. Thus, letting $\widetilde S$ be the market model obtained by removing the $i$-th asset from $S$, it follows that $\operatorname{NA}(\mathcal{P},\widetilde S)$ is satisfied and
    \begin{equation*}
        \exists H \in \R^d \colon X \preccurlyeq [H \Delta S]_{c} \quad \Leftrightarrow \quad \exists H \in \R^{d-1} \colon X \preccurlyeq [H \Delta \widetilde S]_{c},
    \end{equation*} so that $\mathcal{C}(S)=\mathcal{C}(\widetilde S)$.
    Hence, from now on, consider a non-redundant market $S$. Suppose that $\liminf_{n \to \infty} \lVert H_{n} \rVert < \infty$.
    Then there is a subsequence $(H_{n_{k}})_{k\in \N}$ of $(H_n)_{n\in \N}$ which converges to a  vector $H\in \R^d$. In that case,
    \begin{equation*}
        Y_{n_{k}} = Y_{n_{k}} + [H_{n_{k}} \Delta S]_{c} - [H_{n_{k}} \Delta S]_{c} =  [H_{n_{k}} \Delta S]_{c} - X_{n_{k}} \in L^{0}_{c+}
    \end{equation*}
    converges $\mathcal{P}$-q.s.\ to $Y:=[H \Delta S]_{c} - X \in L^{0}_{c+}$ and $X = [H \Delta S]_{c} - Y$,
    so $X\in \mathcal{C}(S)$.
    Now suppose that $\liminf_{n \to \infty} \lVert H_{n} \rVert = \infty$. Let
    \begin{equation*}
        G_{n} := \frac{H_{n}}{1 + \lVert H_{n} \rVert}, \quad  n \in \N.
    \end{equation*}
    As $\lVert G_{n} \rVert \leq 1$, $n\in \N$, there is a subsequence $(G_{n_{k}})_{k\in \N}$ which converges to a vector $G \in \R^{d}$, and $\lVert G \rVert = 1$. Moreover, we have $X_{n_{k}} / (1 + \lVert H_{n_{k}} \rVert) \to 0$ and hence $\mathcal{P}$-q.s.
    \begin{equation*}
        \lim_{k \to \infty} \frac{Y_{n_{k}}}{1 + \lVert H_{n_{k}} \rVert} =  \lim_{k \to \infty} \frac{[H_{n_{k}} \Delta S]_{c}}{1 + \lVert H_{n_{k}} \rVert}- \lim_{k \to \infty} \frac{X_{n_{k}}}{1 + \lVert H_{n_{k}} \rVert} = [G \Delta S]_{c} + 0 = [G \Delta S]_{c}.
    \end{equation*}
    As $Y_{n_{k}} / (1 + \lVert H_{n_{k}} \rVert) \in L^{0}_{c+}$ for all $k \in \N$, it follows that $[G \Delta S]_{c} \in L^{0}_{c+}$. However, by NA($\mathcal{P}$) we must have $[G \Delta S]_{c} = 0$ and thus $G=0$, which contradicts $\lVert G \rVert = 1$. Therefore, $\liminf_{n \to \infty} \lVert H_{n} \rVert = \infty$ is not possible.
\end{proof}

Lemma~\ref{lem:C:clsd} implies the existence of optimal superhedging strategies.

\begin{corollary} \label{cor:opt:strat}
    Suppose that $S\in \mathcal{S}$ satisfies $\operatorname{NA}(\mathcal{P},S)$, and let $X \in L^{\infty}_{c}$. Then, $\pi(X \mid S) > - \infty$ and there exists $H \in \R^{d}$ such that $X \preccurlyeq \pi(X \mid S) + [H \Delta S]_{c}$.
\end{corollary}
\begin{proof}
     Suppose that $\pi(X \mid S) = -\infty$. Then, for all $n \in \N$, there exists $H_{n} \in \R^{d}$ such that $X \preccurlyeq -n + [H_{n} \Delta S]_{c}$ and thus,
    \begin{equation*}
        (X + n) \wedge 1 \preccurlyeq X + n \preccurlyeq [H_{n} \Delta S]_{c}.
    \end{equation*}
    That is, $X_{n} := (X + n) \wedge 1 \in \mathcal{C}(S)$ for all $n \in \N$. Letting $n\to \infty$, Proposition~\ref{lem:C:clsd} implies that $1  \in \mathcal{C}(S)$, which clearly contradicts NA($\mathcal{P},S$).
    If $\pi(X \mid S)$ is finite, then $X_{n} := X - \pi(X \mid S) - 1/n \in \mathcal{C}(S)$ for all $n \in \N$, and thus $X - \pi(X \mid S)\in \mathcal{C}(S)$ by Proposition~\ref{lem:C:clsd}, which implies the existence of $H \in \R^{d}$ such that $X \preccurlyeq \pi(X \mid S) + [H \Delta S]_{c}$.
\end{proof}

\begin{corollary} \label{cor:hedge:dual:Q2}
    Let $\mathfrak{Q} \colon \mathcal{S} \twoheadrightarrow \mathfrak{P}_{c}(\Omega)$ be $\lhd$-monotone and strongly $\operatorname{NA}$-preserving. Moreover, assume that $\bigcup_{Q \in \mathfrak{Q}(S)} \mathfrak{M}^{Q}_{\approx}(S) \subseteq \mathfrak{Q}(S)$, i.e., $Q \in \mathfrak{Q}(S)$, $R \in \mathfrak{M}(S)$, and $R \approx Q$ imply that $R \in \mathfrak{Q}(S)$.
    Then, for every $S$ such that $\operatorname{NA}(\mathcal{P},S)$ is satisfied, $\mathcal{P}\approx \mathfrak{Q}(S)$ and
    \begin{equation*}
        \pi(X \mid S) = \sup_{Q \in \mathfrak{Q}(S)} E_{Q}[X].
    \end{equation*}  
\end{corollary}
\begin{proof}
    Corollary~\ref{cor:hedge:dual:Q} and the additional requirement imply
    \begin{equation*}
        \pi(X \mid S) = \sup_{Q \in \mathfrak{Q}(S)} \sup_{R \in \mathfrak{M}^{Q}_{\approx}(S)} E_{R}[X] = \sup_{R \in \mathfrak{Q}(S)} E_{R}[X].
    \end{equation*}
    Let $A\in \mathcal{F}$. If $c(A)>0$, then Corollary~\ref{cor:opt:strat} and NA($\mathcal{P},S$) imply $\pi(\mathbf{1}_{A} \mid S)>0$, so there must be $Q\in \mathfrak{Q}(S)$ such that $Q(A)=E_Q[\mathbf{1}_{A}] > 0$. Conversely, if $Q(A)>0$, then $\pi(\mathbf{1}_{A} \mid S)>0$. Therefore, $\mathbf{1}_{A}\neq 0$, that is $c(A)>0$, because otherwise $\mathbf{1}_{A}\preccurlyeq 0 + [0\Delta S]_c$ implies that $\pi(\mathbf{1}_{A} \mid S)\leq 0$. 
\end{proof}

According to Lemma~\ref{lem:NAP:NAQ3}, $\mathfrak{M}$, $\mathfrak{NA}$, and, if $\mathcal{P}$ is convex, also $\mathfrak{M}_{\ll}$ and $\mathfrak{M}_{\approx}$ satisfy the requirements of Corollary~\ref{cor:hedge:dual:Q2}. We thus obtain the following version of the robust Superhedging Theorem.

\begin{corollary}
    Consider a market model $S \in \mathcal{S}$ for which $\operatorname{NA}(\mathcal{P}, S)$ holds. Then
    \begin{equation*}
        \pi(X\mid S) = \sup_{Q \in \mathfrak{NA}(S)} E_{Q}[X]= \sup_{Q \in \mathfrak{M}(S)} E_{Q}[X].
    \end{equation*}
    If $\mathcal{P}$ is convex, then also
    \begin{equation}\label{eq:SH}
        \pi(X\mid S) = \sup_{Q \in \mathfrak{M}_{\ll}(S)} E_{Q}[X] = \sup_{Q \in \mathfrak{M}_{\approx}(S)} E_{Q}[X].
    \end{equation}   
\end{corollary}

Note that the first equality in \eqref{eq:SH} corresponds to the robust Superhedging Theorem given in \cite[Theorem 3.4]{BN2015}. 

\begin{theorem} \label{thm:FTAP}
    Suppose that each $P\in \mathcal{P}$ is supported. Moreover, suppose that $\mathfrak{Q} \colon \mathcal{S} \twoheadrightarrow \mathfrak{P}_{c}(\Omega)$ is $\lhd$-monotone, strongly $\operatorname{NA}$-preserving, and satisfies the following conditions:
    \begin{itemize}
        \item $\mathfrak{Q}(S)$ is countably convex for all $S\in \mathcal{S}$, in the sense that if $(\alpha_{n})_{n \in \N}$ is a sequence of non-negative real numbers summing up to $1$ and $Q_{n} \in \mathfrak{Q}(S)$ for all $n \in \N$, then $\sum_{n \in \N} \alpha_{n} Q_{n} \in \mathfrak{Q}(S)$.

        \item $\bigcup_{Q \in \mathfrak{Q}(S)} \mathfrak{M}^{Q}_{\approx}(S) \subseteq \mathfrak{Q}(S) \subseteq \mathfrak{M}(S)$.
    \end{itemize}
    Then, for any $S \in \mathcal{S}$, the following are equivalent:
    \begin{enumerate}
        \item[(i)] $\operatorname{NA}(\mathcal{P},S)$ holds.

        \item[(ii)] For all $P \in \mathcal{P}$ there exists $Q \in \mathfrak{Q}(S)$ such that $P \ll Q$.
    \end{enumerate}
\end{theorem}

\begin{proof}
    (i) $\Rightarrow$ (ii): Suppose that NA($\mathcal{P},S$) holds for $S$ and let $P\in \mathcal{P}$. 
    For all $Q \in \mathfrak{Q}(S)$, consider the Lebesgue decomposition $Q=Q_P+Q_{\perp}$ such that $Q_P\ll P$. Let 
    \begin{equation*}
        u := \sup_{Q \in \mathfrak{Q}(S)} P\bigg(\frac{dQ_P}{dP}>0\bigg).
    \end{equation*}
    Choose $Q^{1}, Q^{2}, \dots \in \mathfrak{Q}(S)$ such that $P(\frac{dQ^n_P}{dP}>0) \to u$ as $n \to \infty$ and set
    \begin{equation*}
        \widetilde{Q} := \sum_{i \in \N} \alpha_{i} Q^{i}
    \end{equation*}
    where $\alpha_i>0$ for all $i\in \N$, and $\sum_{i \in \N} \alpha_{i} = 1$. Then, $\widetilde Q\in \mathfrak{Q}(S)$ by countable convexity, and $\widetilde Q_P= \sum_{i \in \N} \alpha_{i} Q^{i}_P$ and $\widetilde Q_\perp= \sum_{i \in \N} \alpha_{i} Q^{i}_\perp$.  In particular, $\frac{d\widetilde Q_P}{dP}=\sum_{i \in \N} \alpha_{i} \frac{dQ^i_P}{dP}$ and thus
    \begin{equation*}
        P\bigg(\frac{d\widetilde Q_P}{dP}>0\bigg)=P\bigg(\bigcup_{n\in \N}\bigg\{\frac{d Q^n_P}{dP}>0\bigg\}\bigg) \geq \lim_{n\to \infty} P\bigg(\frac{d Q^n_P}{dP}>0\bigg) = u.
    \end{equation*}
    Hence, $P(\frac{d\widetilde Q_P}{dP}>0)=u$. Suppose that $ u < 1$. Then $A:=\{\frac{d\widetilde Q_P}{dP}=0\}\cap S(P)$ satisfies $P(A) > 0$. According to Corollary~\ref{cor:hedge:dual:Q2}, there exists  $\widehat{Q} \in \mathfrak{Q}(S)$ such that $\widehat{Q}(A) > 0$. Note that $\widehat{Q}_P(A)>0$, because one verifies that $\widehat{Q}_\perp(S(P))=0$. Therefore,
    \begin{equation*}
        P\bigg(\bigg\{\frac{d\widehat{Q}_P}{dP} > 0\bigg\}\cap A\bigg) > 0.
    \end{equation*}
    Let $R:=\frac{1}{2}(\widetilde Q + \widehat Q)\in \mathfrak{Q}(S)$. Then
    \begin{equation*}
        P\bigg(\frac{dR_P}{dP}>0\bigg)= P\bigg(\frac12 \bigg(\frac{d\widehat{Q}_P}{dP}+\frac{d\widetilde{Q}_P}{dP}\bigg) > 0\bigg)\geq P\bigg(\bigg\{\frac{d\widehat{Q}_P}{dP} > 0\bigg\}\cap A\bigg)+ P\bigg(\frac{d\widetilde{Q}_P}{dP}>0\bigg) > u,
    \end{equation*}
    which is a contradiction. Hence, $P(\frac{d\widetilde Q_P}{dP} > 0) = 1$ and $P\ll \widetilde Q$.
    
    \smallskip\noindent
    (ii) $\Rightarrow$ (i): Let $H \in \R^{d}$ such that $H \Delta S \geq 0$ $\mathcal{P}$-q.s. Suppose there exists $P \in \mathcal{P}$ such that $P(H \Delta S > 0) > 0$. By assumption, there exists $Q \in \mathfrak{Q}(S)$ such that $P \ll Q$. Consequently, $Q(H \Delta S > 0) > 0$, while  $H \Delta S \geq 0$ $\mathcal{P}$-q.s.\ implies $H \Delta S \geq 0$ $Q$-a.s. But $Q\in \mathfrak{M}(S)$ and thus $E_{Q}[H \Delta S] = 0$, which is absurd. Therefore, $P(H \Delta S > 0) = 0$ for all $P \in \mathcal{P}$, which implies  $[H \Delta S]_c =0$.  
\end{proof}

Note that $\mathfrak{M}$ satisfies the requirements of Theorem~\ref{thm:FTAP}, and, if $\mathcal{P}$ is itself countably convex, so do $\mathfrak{M}_{\ll}$ and $\mathfrak{M}_{\approx}$. Consequently, we obtain Corollaries~\ref{cor:FTAP:M}, \ref{cor:FTAP:Mll}, and \ref{cor:FTAP:Meq}.

\begin{corollary} \label{cor:FTAP:M}
    Suppose that each $P\in \mathcal{P}$ is supported. For any $S \in \mathcal{S}$, the following are equivalent:
    \begin{enumerate}
        \item[(i)] $\operatorname{NA}(\mathcal{P},S)$ holds.

        \item[(ii)] For all $P \in \mathcal{P}$ there exists $Q \in \mathfrak{M}(S)$ such that $P \ll Q$.
    \end{enumerate}
\end{corollary}

The following corollary is a version of \cite[Theorem 3.1]{BN2015}.

\begin{corollary} \label{cor:FTAP:Mll}
    Suppose that each $P\in \mathcal{P}$ is supported and that $\mathcal{P}$ is countably convex. For any $S \in \mathcal{S}$, the following are equivalent:
    \begin{enumerate}
        \item[(i)] $\operatorname{NA}(\mathcal{P},S)$ holds.

        \item[(ii)] For all $P \in \mathcal{P}$ there exists $Q \in \mathfrak{M}_{\ll}(S)$ such that $P \ll Q$.
    \end{enumerate}
\end{corollary}

We further obtain the following stronger version of the Fundamental Theorem of Asset Pricing, where we only consider equivalent martingale measures. A result in that vain can also be found in \cite[Corollary~3.10]{BC2020}.

\begin{corollary} \label{cor:FTAP:Meq}
  Suppose that each $P\in \mathcal{P}$ is supported and that $\mathcal{P}$ is countably convex. For any $S \in \mathcal{S}$, the following are equivalent:
    \begin{enumerate}
        \item[(i)] $\operatorname{NA}(\mathcal{P},S)$ holds.

        \item[(ii)] For all $P \in \mathcal{P}$ there exists $Q \in \mathfrak{M}_{\approx}(S)$ such that $P \ll Q$.
    \end{enumerate}
\end{corollary}

Example~\ref{exmp:supp:nec} shows that the assumption that each $P \in \mathcal{P}$ is supported in Theorem~\ref{thm:FTAP} cannot be dropped.

\begin{example} \label{exmp:supp:nec}
    Consider the unit interval $\Omega = [0, 1]$ equipped with the Borel-$\sigma$-algebra $\mathcal{F} = \mathcal{B}(\Omega)$. Let $\mathcal{P} = \{\delta_{\omega}\mid \omega \in \Omega\} \cup \{\lambda\}$, where $\lambda$ is the Lebesgue measure on $(\Omega, \mathcal{F})$. One verifies that $\lambda$ is not supported. Define $\mathfrak{Q} \colon \mathcal{S} \twoheadrightarrow \mathfrak{P}_{c}(\Omega)$ by
    \begin{equation*}
        \mathfrak{Q}(S) := \bigg\{\sum_{i \in \N} \alpha_{i} \delta_{\omega_{i}} \biggm\vert \forall i \in \N \colon \omega_{i} \in \Omega, \ \alpha_{i} \geq 0 \ \wedge \ \sum_{i \in \N} \alpha_{i} = 1 \quad \text{and} \quad \sum_{i \in \N} \alpha_{i} \Delta S(\omega_{i}) = 0\bigg\}.
    \end{equation*}
    Note that $\mathfrak{Q}(S)\subseteq \mathfrak{M}(S)$. Clearly, $\mathfrak{Q}$ is $\lhd$-monotone. Next, we check that  $\mathfrak{Q}$ is also strongly $\operatorname{NA}$-preserving. Let $S \in \mathcal{S}$ such that NA($\mathcal{P},S$) holds. In this case also NA($\mathcal{P}',S$) is satisfied, where $\mathcal{P}' = \{\delta_{\omega}\mid \omega \in \Omega\}$. By Lemma~\ref{lem:locally:consistent}, there exists $P \in \operatorname{co}(\mathcal{P}')$ such that NA($P,S$) is satisfied. Since $P \in \operatorname{co}(\mathcal{P}')$, there are $n\in \N$, $\omega_{i} \in \Omega$, $\alpha_{i} > 0$, $i=1,\ldots, n$, such that $\sum_{i = 1}^{n} \alpha_{i} = 1$ and $P = \sum_{i = 1}^{n} \alpha_{i} \delta_{\omega_{i}}$. NA($P,S$) implies that there exists $Q \in \mathfrak{M}^{P}_{\approx}(S)$. In particular, $Q = \sum_{i = 1}^{n} \beta_{i} \delta_{\omega_{i}}$ for some $\beta_{i} > 0 $, $i=1,\ldots, n$, with $\sum_{i = 1}^{n} \beta_{i} = 1$.
    Clearly, NA($Q,S$) holds and $Q \in \mathfrak{Q}(S)$. Moreover, as $\mathfrak{Q}(S)\subseteq \mathfrak{M}(S)$, $\mathfrak{Q}(S)$ is strongly $\operatorname{NA}$-preserving. 
    Furthermore, $\mathfrak{Q}(S)$ is countably convex and  $\bigcup_{Q \in \mathfrak{Q}(S)} \mathfrak{M}^{Q}_{\approx}(S) \subseteq \mathfrak{Q}(S) \subseteq \mathfrak{M}(S)$. However, there is clearly no $Q \in \mathfrak{Q}(S)$ such that $\lambda \ll Q$, and Theorem~\ref{thm:FTAP} does not hold. \hfill $\diamond$
\end{example}

% \bibliographystyle{abbrv}
% \bibliography{references}
\printbibliography

\end{document}